\documentclass[13pt]{amsart}
\usepackage{amsmath}
\usepackage{}
\usepackage{amsfonts}
\usepackage{amsfonts,latexsym,rawfonts,amsmath,amssymb,amsthm}
\usepackage[plainpages=false]{hyperref}

\usepackage{pdfsync}

\usepackage[a4paper]{geometry}
\usepackage{esint}
\usepackage{graphics,graphicx}
\usepackage{tikz}

\numberwithin{equation}{section}
\newcommand{\beq}{\begin{equation}}
\newcommand{\eeq}{\end{equation}}
\newcommand{\beqs}{\begin{eqnarray*}}
\newcommand{\eeqs}{\end{eqnarray*}}
\newcommand{\beqn}{\begin{eqnarray}}
\newcommand{\eeqn}{\end{eqnarray}}
\newcommand{\beqa}{\begin{array}}
\newcommand{\eeqa}{\end{array}}

\def\lra{\longrightarrow}

\def\bc{\begin{center}}
\def\ec{\end{center}}

\def\begeq{\begin{equation}}
\def\endeq{\end{equation}}
\def\and{\quad{\rm and}\quad}

\let\lra=\longrightarrow

\def\mapright\#1{\,\smash{\mathop{\lra}\limits^{\#1}}\,}

\newtheorem{prop}{Proposition}[section]
\newtheorem{theo}[prop]{Theorem}
\newtheorem{lem}[prop]{Lemma}

\newtheorem{cor}[prop]{Corollary}
\newtheorem{rem}[prop]{Remark}
\newtheorem{ex}[prop]{Example}

\newtheorem{conj}[prop]{Conjecture}

\begin{document}

\title{Tian's $\alpha_{m,k}^{\hat K}$-invariants on group compactifications}

%\date{-}

\author{Yan ${\rm Li}^*$}
\author{Xiaohua  ${\rm Zhu}^{**}$}

\address{* School of Mathematics and Statistics, Beijing Institute of Technology, Beijing, 100081, China.}
\address{** School of Mathematical Sciences, Peking
University, Beijing, 100871, China.}

\email{liyanmath@pku.edu.cn \ \ \ xhzhu@math.pku.edu.cn}

\thanks {* Partially supported by China Post-doctoral Grant BX20180010.}

\thanks {** Partially supported by  NSFC Grants 11771019 and BJSF Grants Z180004.}

\begin{abstract}
In this paper, we compute  Tian's  $\alpha_{m,k}^{K\times K}$-invariant on a polarized  $G$-group compactification, where $K$ denotes a maximal compact subgroup of a connected complex reductive group $G$.  We  prove that  Tian's  conjecture (see Conjecture \ref{Tian-conj} below) is  true for  $\alpha_{m,k}^{K\times K}$-invariant on such manifolds when $k=1$, but it fails  in general   by producing counter-examples when $k\ge 2$.
\end{abstract}

\maketitle

\section{Introduction}

Let $(M,L)$ be an $n$-dimensional polarized complex manifold.  Namely,  there is a Hermitian metric $h$ on $L$ such that its curvature $\omega_h$ gives a K\"ahler metric on $M$.  Then  for any $m\in\mathbb N_+$ and the $m$-th power $L^m$ of $L$,   there is  an inner product $\ll,\gg_{h^m}$ on $H^0(M,L^m)$ defined in terms of $h$ and $\omega_h$ as follows,
\begin{eqnarray}
\ll s_1,s_2\gg_{h^m}=\int_M (s_1,s_2)_{h^m}\omega_h^n,~\forall~ s_1,s_2\in H^0(M,L^m).\notag
\end{eqnarray}

Let $Gr_k(H^0(M,L^m))$ be the Grassmannian manifold of $k$-dimensional subspaces of $H^0(M,L^m)$.  Then for any $k$  linearly  independent sections $s_1,...,s_k\in H^0(M,L^m)$,  we have
$$\Pi={\rm Span}_{\mathbb C}\{s_1,...,s_k\}  \in Gr_k(H^0(M,L^m)).$$
Choose an orthonormal basis in $\Pi$, for simplicity,   $s_1,...,s_k $, with respect to $\ll,\gg_{h^m}$ such that
\begin{align} \ll s_i,s_j\gg_{h^m}=\delta_{ij}.\notag \end{align}
Set
\begin{eqnarray}\label{partial Bergeman kernel}
b_{L^m, h^m}^\Pi(x)=\sum_{i=1}^k |s_i|_{h^m}^2(x),
\end{eqnarray}
and
\begin{eqnarray}\label{alpha Pi}
\alpha_{m,k}^\Pi=\sup\left\{\alpha\left|\int_M\left(b_{L^m,h^m}^\Pi(x)\right)^{-\frac\alpha{m}}\omega_h^n<\infty\right.\right\}.\notag
\end{eqnarray}
It is known that $\alpha_{m,k}^\Pi$ is  independent  of  the choice of basis $\{s_i\}_{i=1}^k$ and Hermitian metric $h$ (cf. \cite[Appendix]{Tian297}).

Tian's $\alpha_{m,k}$-invariant  is defined   on $(M,L)$ by
\begin{eqnarray}\label{alpha m-k def}
\alpha_{m,k}(M,L)=\inf\left\{\alpha_{m,k}^\Pi |~\Pi\in Gr_k(H^0(M,L^m)) \right \}.
\end{eqnarray}
In  \cite[Appendix]{Tian297},    Tian proposed the following conjecture.

\begin{conj}\label{Tian-conj}
For any polarized manifold $(M,L)$ and $k\in{\mathbb N}_+$, there exists an $m_0\in{\mathbb N}_+$ such that for all $m\in{\mathbb N}_{\geq m_0}$,
$$\alpha_{m,k}(M,L)=\alpha_{m_0,k}(M,L).$$
\end{conj}

Recently, by using tools from the minimal model program,  Birkar \cite{Birkar2016} made a significant process towards Tian's conjecture in case of $k=1$ for the anti-canonical
line bundle $K_M^{-1}$  on a Fano manifold $M$ and he proved the following quantization  theorem:

\begin{theo}\label{birkar}
Let $M$ be a Fano manifold.  Suppose that  $\alpha(M,K_M^{-1})$-invariant is strictly less than $1$.  Then
there exists an  $m_0\in\mathbb N_+$ such that for any $l\in \mathbb N_+$ it holds
$$\alpha_{m_0l,1}(M,K_M^{-1})=\alpha(M,K_M^{-1}).$$
\end{theo}

The $\alpha$-invariant (or $\alpha^{\hat K}$-invariant,   see below)  was introduced by Tian  for any K\"ahler class in 1987  \cite{T87}.
 For a  polarized K\"ahler class $[\omega_h]=2\pi  c_1(M, L)$,    $\alpha$-invariant is defined by
 \begin{align}\label{alpha-inv} \alpha(M,L)=\sup\{\alpha|\int_Me^{-\alpha(\varphi-\sup_M\varphi)}\omega_h^n\leq C_\alpha,~\forall~\varphi\in C^\infty(M)\text{ with }\omega_h+\sqrt{-1}\partial\bar\partial\varphi>0\}.
 \end{align}
  Usually we denote by $\alpha(M)=\alpha(M, K_M^{-1})$ the $\alpha$-invariant for the first Chern class $2\pi c_1(M)$ on  a Fano manifold.

By estimating  $\alpha(M)$-invariant  Tian solved the existence problem of K\"ahler-Einstein metrics on  Fano surfaces \cite{T90}. A complete computation of $\alpha(M)$-invariant on Fano surfaces was given by  Cheltsov \cite{Ch08} ten years ago. For higher dimensional  Fano manifolds, Song computed $\alpha^{\hat K}(M)$-invariant of toric Fano manifolds,   where  $\hat K$  is  a  compact subgroup
of holomorphic transformation group  ${\rm Aut}(M)$ which  generated by the maximal compact torus  $T$ and the corresponding Weyl group of the maximal reductive subgroup ${\rm Aut}_r(M)$ of ${\rm Aut}(M)$ with respect to $T$ \cite{SongAJM} (also see Remark \ref{rem1}).  Delcroix generalized Song's result to a $G$-group compactification by the study of the Newton bodies  of convex potentials of  associated  metrics  \cite{Del1}, provided  that the compactification is  Fano. Here $G$ is a connected complex reductive Lie group. % which is the complexification of a compact Lie group $K$.
There are recent developments on  K\"ahler-Einstein metrics on Fano group compactifications, we refer the reader to \cite{Del1,  LZZ, LTZ2}, etc.

It is known that the $\alpha^{\hat K}$-invariant is more effective  than  $\alpha$-invariant  in the study of  K\"ahler-Einstein metrics on Fano manifolds, for examples in papers of \cite{ T87, TY, T90,  BS, DemaillyKollar},  etc.
 Demailly  also proved  that
 \begin{eqnarray}\label{alpha K}
\alpha^{\hat K}(M,L)=\lim_{m\to\infty} \alpha_{m,1}^{\hat K}(M,L),
\end{eqnarray}
where  $\alpha^{\hat K}(M,L)$ denotes the $\alpha$-invariant for $\hat K$-invariant K\"ahler potentials associated to $L$  as in  (\ref{alpha-inv})  (cf. \cite[Appendix]{Cheltsov-Shramov2008}).

However, to the  best  of   our   knowledge, there are very few references for the computation of  $\alpha_{m,k}$-invariants, except  on Del Pezzo surfaces \cite{T90,  Shi10, Cheltsov-Kosta14}. On the other hand,
for a compact subgroup $\hat K$ of ${\rm Aut}(M)$,  one  can define  $\alpha_{m,k}^{\hat K}$-invariant by considering only $\hat K$-invariant subspaces  $\Pi={\rm Span}_{\Bbb C}\{s_1,...,s_k\}\subseteq H^0(M,L^m)$  as in (\ref{alpha Pi}). Namely, we have
 \begin{eqnarray}\label{alpha m-k-K def}
\alpha_{m,k}^{\hat K}(M,L)=\inf\left\{\alpha_{m,k}^\Pi |~\Pi\in Gr^{\hat K}_k(H^0(M,L^m)) \right \},
\end{eqnarray}
where $Gr^{\hat K}_k(H^0(M,L^m))$ is the subset of $Gr_k(H^0(M,L^m))$ which consists of $\hat K$-invariant subspaces of dimension $k$ in  $H^0(M,L^m)$  (cf. \cite[p.156]{Tian297}).

Let $G$ be  a connected complex reductive Lie group as above and   $K$ a maximal compact subgroup of $G$. %Suppose that $G$ is the complexification of a connected compact Lie group $K$.
%Recall the definition of $G$-group compactification.
We call a compact complex manifold $M$ a {\it (bi-equivariant) compactification of $G$} (or \emph{$G$-group compactification} for simplicity) if it admits a holomorphic $G\times G$-action with an open and dense orbit isomorphic to $G$ as a $G\times G$-homogeneous space.  $(M, L)$ is called a {\it polarized compactification} of $G$   if  $L$ is a $G\times G$-linearized ample line bundle on $M$. For more knowledge  and examples of $G$-group compactification,  we refer the reader to \cite{Timashev-Sbo, AK, Del2, Del3}, etc.

In this paper, we give a way to compute  $\alpha_{m,k}^{K\times K}$-invariant  for any $m,k$ on  a polarized $G$-group compactification $(M,L)$ in terms of the data of its associated polytope, see  Theorem \ref{group alpaha m-k thm general}.   As an application of Theorem \ref{group alpaha m-k thm general}, we prove a version of Birkar's theorem (Theorem \ref{birkar}) for $\hat K$-invariant K\"ahler potentials on such a manifold  as follows.\footnote{We need not assume that the $G$-group compactification is Fano.}

\begin{theo}\label{alpha m-1 K-K prop}
Let $(M,L)$ be a polarized compactification of $G$. Then there exists an $m_0\in\mathbb N_+$ such that
\begin{eqnarray}
%\label{alpha m-1 K-K}
\alpha^{K\times K}_{m_0 l,1}(M,L)=\alpha^{K\times K}(M,L),\forall ~l\in\mathbb N_+.\notag
\end{eqnarray}
\end{theo}

Theorem \ref{alpha m-1 K-K prop} confirms Conjecture \ref{Tian-conj} for $\alpha^{K\times K}_{m_0l,1}$-invariant on polarized  group compactifications. For general $k\ge 2$,   we have the following  criterion of Conjecture \ref{Tian-conj} for $\alpha_{m,k}^{\hat K}$-invariant on  toric  Fano  manifolds.

Let    $M$  be   an $n$-dimensional   toric  Fano  manifold and $P$ be  the associated polytope  of $(M,K_M^{-1})$ (cf. \cite{CLS}). We define  a  function $t(x)$  on $P$    by
\begin{eqnarray}\label{t(x)}
t(x)=\sup\left\{t\in(0,1)\left|\frac{t}{t-1}x\in P\right.\right\}.
\end{eqnarray}
 Then we have

\begin{theo}\label{weak conj thm}Let $T$ be the  maximal compact torus  on the toric  Fano  manifold $M$  as above.
Then for any fixed $k\in\mathbb N_+$, there exists an $m_k\in\mathbb N_+$ such that
\begin{eqnarray}
\label{weak conj}
\alpha_{m_kl,k}^{T}(M,K_M^{-1})=\alpha_{m_k,k}^{T}(M,K_M^{-1}),~\forall~ l\in\mathbb N_+,
\end{eqnarray}
if and only if there is a facet  $\mathcal F$ of the associated polytope $P$  such that
\begin{eqnarray}\label{number count}
t(x)|_{\mathcal F}=\alpha^{T}(M)\text{ and }\#\left(\frac1{m_k}\mathfrak M\cap\mathcal F\right)\geq k,
\end{eqnarray}
where   $\mathfrak M$ is  the   character lattice of  Lie algebra of  $T^\mathbb C$.
Moreover,
$$\alpha_{m_k,k}^{T}(M,K_M^{-1})= \alpha^{T}(M) $$
if  (\ref{number count}) holds.   Otherwise, if  there is no  facet  $\mathcal F$ of  $P$ and    $m_k\in\mathbb N_+$ such that
 (\ref{number count}) holds, then
$$\alpha_{m,k}^{T}(M,K_M^{-1})>\alpha^{T}(M),~ \forall ~m\in\mathbb N_+.
$$
\end{theo}

Using  Theorem \ref{weak conj thm},   we can  find   counter-examples of toric manifolds to  show that Conjecture \ref{Tian-conj}  fails in general for $\alpha_{m,k}^T$-invariant  when $k\ge 2$ (cf. Example \ref{ex}-\ref{ex-Pn}). %sections space   (cf. Example \ref{ex}-\ref{ex-Pn}).

 In the paper, we also prove  other results by computing  $\alpha_{m,k}^{\hat K}$-invariant, such as Corollary \ref{alpha-KtimesK}, Proposition \ref{group alpha del thm} and Proposition \ref{alpha m-k limit prop}, which cover or generalize the corresponding results in \cite{Del1}, \cite{Cheltsov-Shramov2008}, \cite{Blum-Jonsson17} and  \cite{SongAJM}, see Remark \ref{rem37},   Remark \ref{alpha-1} and Remark \ref{rem1}, respectively.

The paper is organized as follows.   In Section \ref{Sect Bergeman}, we  first decompose $H^0(M,L^m)$ into direct sum of $K\times K$-invariant subspaces, then we reduce $b_{L^m,k}^\Pi(\cdot)$ to a convex function on  a toric submanifold $Z$ of $M$ under the log-affine coordinates (cf. \eqref{Bergeman pi toric} below). In Section \ref{sect computation},  we give a computation for  all $\alpha^{K\times K}_{m,k}$ in Theorem \ref{group alpaha m-k thm general} and prove Theorem \ref{alpha m-1 K-K prop}. Section 4 is devoted  to toric Fano manifolds,  where  Theorem \ref{weak conj thm} is proved.  Some examples of polarized $G$-group compactifications with explicit computations  of  $\alpha^{K\times K}_{m,k}$-invariant  are given  in  Section 5.   Section 6 is an appendix where we give a direct proof of Lemma \ref{determinent}.

\vskip5mm

\noindent {\bf Acknowledgements.} The authors would like to thank Professor Gang Tian for inspiring conversations on the paper. %and Professor D. A. Timash\"ev for introducing to us his paper \cite{Timashev-Sbo}.
 They also thank referee for valuable comments, especially on Lemma \ref{determinent}.
%Yan Li would also like to thank Professor D. A. Timash\"ev for explaining to him \cite[Theorem 9]{Timashev-Sbo}, which plays %important role in proof of Lemma \ref{determinent}.

\section{Bergman Kernel   on  $G$-group compactifications}\label{Sect Bergeman}

In this paper, we always assume that $G$ is a connected complex reductive Lie group which is the complexification of a compact Lie group $K$. Let $T^\mathbb C$ be a maximal complex torus of $G$ with Lie algebra $\mathfrak t^{\mathbb C}$ and $\mathfrak M$ the   character lattice  of $\mathfrak t^{\mathbb C}$. Let $J$ be the complex structure of $G$. We set $\mathfrak a=J\mathfrak t$,
where $ \mathfrak t$ is the Lie algebra of the maximal compact torus $T$.

Denote  the root system and the Weyl group of $G$ with respect to $T^\mathbb C$ by $\Phi$ and $W$, respectively. Choose a system of positive roots $\Phi_+$.  Then it defines a positive Weyl chamber $\mathfrak a_+\subseteq\mathfrak a$,\footnote{When $G=T^\mathbb C$, we set $\mathfrak a_+=\mathfrak a$.}
$$\mathfrak a_+:=\{x\in\mathfrak a|\alpha(x)>0,~\forall\alpha\in\Phi_+\}.$$
The closure $\overline{\mathfrak a_+}$ is isomorphic to the quotient space  $\mathfrak a/ W$.

%Recall the definition of $G$-group compactification.  We call a compact complex manifold $M$ a {\it (bi-equivariant) compactification of $G$} (or \emph{$G$-group compactification} for simplicity) if it admits a holomorphic $G\times G$ action with an open and dense orbit isomorphic to $G$ as a $G\times G$-homogeneous space.  $(M, L)$ is called a {\it polarized compactification} of $G$   if  $L$ is a $G\times G$-linearized ample line bundle on $M$. For more knowledge  and examples of $G$-group compactification,  we refer the reader to see \cite{Timashev-Sbo, AK, Del2, Del3}, etc..

\subsection{$K\times K$-invariant K\"ahler metrics}
Let  $Z$ be the closure of a maximal torus $T^\mathbb C$ in  $G$-group compactification $M$. It is known that $(Z, L|_Z)$ is a polarized toric manifold with a $W$-action, and $L|_Z$ is a $WT^\mathbb C$-linearized ample toric line bundle on $Z$  \cite[Theorem 2.4]{AK} (see also \cite{Timashev-Sbo, AB1, AB2} for detail proofs).

By the $KAK$-decomposition \cite[Theorem 7.39]{Kna}, for any $g\in G$,
there are $k_1,\,k_2\in K$ and $x\in\mathfrak a$ such that $g=k_1\exp(x)k_2$. Here $x$ is uniquely determined up to the $W$-action. This means that $x$ is unique in $\overline{\mathfrak a_+}$.
Thus any $K\times K$-invariant function $\Psi$ on $G$ descends to a $W$-invariant function $\psi$ on $\mathfrak a$ by the relation
\begin{align}\label{0222-01}
\Psi(\exp(x))=\psi(x),~\forall x\in{\mathfrak a}.
\end{align}
We usually call  $\psi$ \emph{the function associated to $\Psi$}. Conversely, given a $W$-invariant $\psi$ on $\mathfrak a$, we can define  a $W$-invariant function  $\Psi$   on $T^\mathbb C$  by \eqref{0222-01}. By using the $KAK$-decomposition, we can extend it to a $K\times K$-invariant function on $G$ by
$$\Psi(k_1\exp(x)k_2):=\Psi(\exp(x)),~\forall k_1,k_2\in K\text{ and }x\in\mathfrak a.$$
Clearly $\Psi$ is well-defined since $\psi$ is $W$-invariant.

Let $\omega_0\in 2\pi c_1(L)$ be  a $K\times K$-invariant K\"ahler form on $(M, L)$.
By the $K\times K$-invariance,
the restriction of $\omega_0$ on $Z$ is a toric K\"ahler metric. Thus it induces a smooth, strictly convex function $\psi$ on ${\mathfrak a}$, which is $W$-invariant such that $\omega_0|_{T^\mathbb C}=\sqrt{-1}\partial\bar\partial \psi$ \cite{AL}.\footnote{When $G$ is not semi-simple, $\psi$ is determined up to an affine function $l_\xi(x)=\xi_ix^i$ whose gradient $\xi$ lie in the center $\mathfrak z(\mathfrak g)$. We may choose $\xi$ such that the image of $\nabla \psi$ is $2P$.} Conversely, it can be verified  that if $\psi$ is $W$-invariant, strictly convex and smooth, then $\omega=\sqrt{-1}\partial\bar\partial \Psi$ is a K\"ahler metric on $G$  (cf.   \cite[Theorem 1.2]{Del2}). From now on, we will not distinguish $\Psi$ and $\psi$ for simplicity.

The following $KAK$-integral formula can be found in \cite[Proposition 5.28]{Kna2}.

\begin{prop}\label{KAK int}
Let $dV_G$ be a Haar measure on $G$. Denote by $dx$  the Lebesgue measure on $\mathfrak{a}$, which is normalized by the lattice of one parameter subgroups of $T^\mathbb C$.
Then there exists a constant $C_H>0$ such that for any $K\times K$-invariant, $dV_G$-integrable function $\Psi$ on $G$,
$$\int_G \Psi(g)\,dV_G= C_H\int_{\mathfrak{a}_+}\psi(x)\mathbf{J}(x)\,dx,$$
where
\begin{align}\label{J}\mathbf J(x)=\prod_{\alpha \in \Phi_+} \sinh^2\alpha(x).\end{align}
\end{prop}

\subsection{ Irreducible decomposition of  $H^0(M,L^m)$}

Let $P$ be the associated polytope of $(M,L)$, which is defined as the one of the toric manifold $(Z,L|_Z)$  (cf. \cite{AK, Del2, LZZ}).  Then $P$  is a $W$-invariant convex lattice polytope in $\mathfrak a^*$, where  $\mathfrak a^*$ is the dual of $\mathfrak a$. Choose a $W$-invariant Killing inner product  on $\mathfrak a$ as in \cite{Del2,LZZ}.  Let  $\mathfrak a_+^*$ be the dual of $\mathfrak a_+$ under this inner product and $P_+=P\cap\overline{\mathfrak a_+^*}$. It is proved in \cite[Section 2.2]{AB2} (see also \cite[Section 2]{AK}) that as a $G\times G$-representation,
\begin{eqnarray}\label{H^0}
H^0(M,L)=\oplus_{\lambda\in P_+\cap\mathfrak M}\text{End}(V_\lambda),
\end{eqnarray}
where $V_\lambda$ is the irreducible representation of $G$ with highest weight $\lambda$.

%Since the group $T^\mathbb C\times T^\mathbb C$  naturally acts on $H^0(M,L^m)$,  each
The restriction of $V_\lambda$ on $T^\mathbb C$ can be decomposed as %\cite[Theorem 5.75]{Kna},
\begin{eqnarray}\label{toric decom}
V_\lambda=\oplus_{\mu\in\mathfrak M}n_\lambda(\mu)\check V_\mu,
\end{eqnarray}
where $\check V_\mu$ are  one-dimensional  irreducible representations  of $T^\mathbb C$ with weight $\mu$ and  multiplicity $n_\lambda(\mu)$ given by Kostant's multiplicity formula \cite[Section 124, 2$^\circ$]{Zhelobenko}. In particular,
\begin{eqnarray}\label{W-inv multiplicity}
n_\lambda(w(\mu))=n_\lambda(\mu),\forall~ w\in W\text{ and }\mu\in\mathfrak M,
\end{eqnarray}
\begin{eqnarray*}\label{dominant multiplicity}
n_\lambda(\lambda)=1\text{ and }
n_\lambda(\mu)=0, \forall \mu\not\in\text{Conv}(\{w(\lambda)|~w\in W\}),
\end{eqnarray*}
where $\text{Conv}(\{w(\lambda)|~w\in W\})$ is  the convex hull of $\{w(\lambda)|~w\in W\}$. Thus by \eqref{H^0} and \eqref{toric decom}, we can rewrite $H^0(M,L^m)$ as a  $T^\mathbb C\times T^\mathbb C$-representation as following,
\begin{eqnarray}\label{H^0 toric}
H^0(M,L^m)=\oplus_{\lambda,\mu\in mP\cap\mathfrak M}n_{\lambda\mu}\check V_\lambda\otimes\check V_\mu^*
\end{eqnarray}
for some  multiplicities $n_{\lambda\mu}\in\mathbb N$.

Let $h$ be any $K\times K$-invariant Hermitian metric on $L$.  Then  the inner product $\ll\cdot,\cdot\gg_{h^m}$ is
 $T\times T$-invariant.
  The  following lemma is essentially due to  \cite[Section 20, Corollary 3]{Zhelobenko}.

\begin{lem}\label{orthogonal}   Let  $\check V_{\lambda_1}\otimes \check V_{\lambda_2}^*$ and  $ \check V_{\mu_1}\otimes \check V_{\mu_2}^*$ be two sub-representations in (\ref{H^0 toric}) with respect to the pairs $(\lambda_1,\lambda_2)$ and  $(\mu_1,\mu_2)$, respectively.
Suppose that  $(\lambda_1,\lambda_2)\not=(\mu_1,\mu_2)$.  Then  for any $s_{\lambda_1\lambda_2}\in n_{\lambda_1\lambda_2} \check V_{\lambda_1}\otimes \check V_{\lambda_2}^*$ and $s_{\mu_1\mu_2}\in n_{\mu_1\mu_2} \check V_{\mu_1}\otimes \check V_{\mu_2}^*$, it holds
 $$\ll s_{\lambda_1\lambda_2},s_{\mu_1\mu_2}\gg_{h^m}=0.$$
\end{lem}

\begin{proof} We give a brief proof for completeness.
Let $\{s_{\lambda_1\lambda_2}^i\}_{i=1}^{n_{\lambda_1\lambda_2}}$ be a basis of $n_{\lambda_1\lambda_2} \check V_{\lambda_1}\otimes \check V_{\lambda_2}^*$ and $\{s_{\mu_1\mu_2}^a\}_{a=1}^{n_{\mu_1\mu_2}}$ be a basis of $n_{\mu_1\mu_2} \check V_{\mu_1}\otimes \check V_{\mu_2}^*$, respectively. Consider an $(n_{\lambda_1\lambda_2}\times n_{\mu_1\mu_2})$-matrix $\mathbf{H}^{(1)}=(H^{(1)}_{ia})$,
$$H^{(1)}_{ia}=\ll s_{\lambda_1\lambda_2}^i,s_{\mu_1\mu_2}^a\gg_{h^m},~1\leq i\leq n_{\lambda_1\lambda_2},~1\leq a\leq n_{\mu_1\mu_2},$$
and an $(n_{\mu_1\mu_2}\times n_{\mu_1\mu_2})$-matrix $\mathbf{H}^{(2)}=(H^{(2)}_{ab})$,
$$H^{(2)}_{ab}=\ll s_{\mu_1\mu_2}^a,s_{\mu_1\mu_2}^b\gg_{h^m},~1\leq a,b\leq n_{\mu_1\mu_2}.$$
Define an $(n_{\lambda_1\lambda_2}\times n_{\mu_1\mu_2})$-matrix $\mathbf{H}=(H^a_{i})$,
$$H^a_{i}=\sum_{b=1}^{n_{\mu_1\mu_2}}H^{(2);ab}H^{(1)}_{ib},$$
where ${\mathbf{H}^{(2)}}^{-1}=(H^{(2);ab})$ is the inverse of $\mathbf{H}^{(2)}$. Then  $\mathbf{H}$ gives a homeomorphism
$$\mathbf{H}:n_{\lambda_1\lambda_2} \check V_{\lambda_1}\otimes \check V_{\lambda_2}^*\to n_{\mu_1\mu_2} \check V_{\mu_1}\otimes \check V_{\mu_2}^*.$$
Note that $\ll\cdot,\cdot\gg_{h^m}$ is $T\times T$-invariant underaction. Thus $\mathbf H$ is $T\times T$-equivariant. More precisely,
$$(t_1,t_2)\circ \mathbf H=\mathbf H\circ (t_1,t_2),~\forall (t_1,t_2)\in T\times T.$$
 Since both of $\check V_{\lambda_1}\otimes \check V_{\lambda_2}^*$ and $\check V_{\mu_1}\otimes \check V_{\mu_2}^*$ are irreducible $T\times T$-representations with
$$\check V_{\lambda_1}\otimes \check V_{\lambda_2}^*\not\cong\check V_{\mu_1}\otimes \check V_{\mu_2}^*,$$
by Schur' lemma, it must hold $$\mathbf {H}=O$$
is the zero homomorphism. Consequently,
$$\ll s_{\lambda_1\lambda_2}^i,s_{\mu_1\mu_2}^a\gg_{h^m}=H^{(1)}_{ia}=H^b_iH^{(2)}_{ab}=0.$$
Hence we prove the lemma.
\end{proof}

\subsection{Estimate of $b_{L^m, h^m}(x)$ and $b^\Pi_{L^m, h^m}(x)$}

 Let $$\cup_{\lambda_1,\lambda_2\in mP\cap\mathfrak M}\cup_{i=1}^{n_{\lambda_1\lambda_2}}\{s^i_{\lambda_1\lambda_2}\}$$
 be an orthonormal basis of $H^0(M,L^m)$ chosen as in Lemma \ref{orthogonal}.

Then  the Bergman kernel is  given by
\begin{eqnarray}\label{Bergeman}
b_{L^m,h^m}(x)=\sum_{\lambda\in mP_+\cap\mathfrak M}\sum_{s^i_{\lambda_1\lambda_2}\in\text{End}(V_\lambda)}\frac{\|s^i_{\lambda_1\lambda_2}(x)\|^2_{h^m}}{\ll s^i_{\lambda_1\lambda_2},s^i_{\lambda_1\lambda_2}\gg_{h^m}}.
\end{eqnarray}
%We note that the Bergman kernel $b_{L^m,h^m}(x)$ is $K\times K$-invariant. This is because that $h$ is $K\times K$-invariant, and $K\times K$ then acts as a subgroup of unitary group on the inner product space $(H^0(M,L^m),\ll\cdot,\cdot\gg_{h^m})$.

Assume that $L^{m_0}$ is very ample so that $M$ can be embedded into a projective space $\mathbb CP^{h^0-1}$ by $|L^{m_0}|$ for $h^0=\dim(H^0(M,L^{m_0}))$.  Then we can  choose the $K\times K$-invariant Hermitian metric $h$ on $L$  to be the one induced by the Fubini-Study metric of  $\mathbb CP^{h^0-1}$.
%as in \cite[Proposition 4.1]{SongAJM}.
To write down this metric, let us first fix a trivialization $s_0$ of $L$ on $G$. Denote by $L_{e}$ the fiber at the identity $e\in G$ and $0\not=s_{e}\in L_{e}$. Since $L$ is $G\times G$-linearized, we can define
\begin{align}\label{trivialization}
s_0(x)=(x,e)s_{e},~\forall x\in G.
\end{align}
as a trivialization on $G(\subseteq M)$. For $L^m (m\in\mathbb N_+)$, we chose its $m$-th power $s_0^m$ as a background trivialization. Thus, there are holomorphic functions $\{\chi^{i}_{\lambda_1\lambda_2}\}$ on $G(\subseteq M)$ such that
$$s^{i}_{\lambda_1\lambda_2}(x)=\chi^{i}_{\lambda_1\lambda_2}(x)s_0(x),~\forall x\in G(\subseteq M).$$
The Hermitian metric $h$ with respect to $s$ is then given by
\begin{eqnarray}\label{Fubini-Study potential}
h(x)=(\sum_{\lambda_1,\lambda_2\in m_0P\cap\mathfrak M}\sum_{i=1}^{n^0_{\lambda_1\lambda_2}}|\chi^{i}_{\lambda_1\lambda_2}(x)|^2)^{-\frac 1{m_0}}.
\end{eqnarray}
By \eqref{H^0}, $h$ can be rewritten as
$$h(x)=(\sum_{\lambda\in m_0P_+\cap\mathfrak M}\sum_{s^{i}_{\lambda_1\lambda_2}\in\text{End}(V_\lambda)}|\chi^{i}_{\lambda_1\lambda_2}(x)|^2)^{-\frac 1{m_0}}.$$
Note that by the above embedding of $M$ into $\mathbb CP^{h^0-1}$, for any $x\in G$, the function $\chi^{i}_{\lambda_1\lambda_2}(x)$ is just the $(\lambda_1,\lambda_2)$-entry of $x$, which is represented in some $\text{End}(V_\lambda)$ as a matrix in $GL(V_\lambda)$.
By \eqref{H^0}, The group $G\times G$ acts as a subgroup of $\oplus_{\lambda\in m_0P_+\cap\mathfrak M}GL(\text{End}(V_\lambda))$ and $K\times K$ acts as a subgroup of the unitary group $\oplus_{\lambda\in m_0P_+\cap\mathfrak M}U(\text{End}(V_\lambda))$. Since $K\times K$ acts as multiplying $x$ by a unitary metric, $h(x)$ is $K\times K$-invariant.

Then we estimate the Bergman kernel $b_{L^m,h^m}(x)$ defined by \eqref{Bergeman}. Under \eqref{trivialization}, it is reduced to
\begin{align}\label{Bergeman-coordinate}
b_{L^m,h^m}(x)
=&\frac1{h^m(x)}\sum_{\lambda\in mP_+\cap\mathfrak M}\sum_{s^{i}_{\lambda_1\lambda_2}\in\text{End}(V_\lambda)}\frac{|\chi^{i}_{\lambda_1\lambda_2}(x)|^2}{\ll s^i_{\lambda_1\lambda_2},s^i_{\lambda_1\lambda_2}\gg_{h^m}}\notag\\
\simeq&\frac1{h^m(x)}\sum_{\lambda\in mP_+\cap\mathfrak M}\sum_{s^{i}_{\lambda_1\lambda_2}\in\text{End}(V_\lambda)}|\chi^{i}_{\lambda_1\lambda_2}(x)|^2=:\hat b_{L^m,h^m}(x).
\end{align}
Here and after, we always write $f\simeq g$ if two quantities (functions, integrations etc.) $f,g$ satisfy
$$0<c^{-1}\leq\left|\frac f g\right|\leq c$$
for some constant $c>1$. It is direct to see that $\hat b_{L^m,h^m}(x)$ is also $K\times K$-invariant.

By $K\times K$-invariance and the relation \eqref{0222-01}, it suffices to compute the restriction of $h(x)$ and $\hat b_{L^m,h^m}(x)$ on $T^\mathbb C$. We have
$$\chi^i_{\lambda\mu}(\cdot)|_{T^\mathbb C}=0,~ \forall \lambda\not=\mu.$$
Moreover,  by equivariance,  it holds that for any $e^{x+\theta\sqrt{-1}}\in T^\mathbb C$ under the log-affine coordinates,
$$\chi_{\lambda\lambda}^i(e^{x+\theta\sqrt{-1}}\cdot p)=e^{\lambda(x+\theta\sqrt{-1})}\chi_{\lambda\lambda}^i(p),$$
where $\lambda(x)=\sum_{i=1}^r\lambda_ix^i$ is a linear function. Thus by \eqref{toric decom}, \eqref{W-inv multiplicity} and \eqref{Fubini-Study potential}, the restriction of $h$ on the torus $T^\mathbb C\subseteq Z$ can  be written as
\begin{eqnarray}\label{hermitian toric}
h|_{T^\mathbb C}(x)=\left(\sum_{\lambda\in m_0P\cap\mathfrak M}\bar n(\lambda)e^{2\lambda(x)}\right)^{-\frac1{m_0}},~\forall x\in\mathfrak a,
\end{eqnarray}
where
$\bar n(\lambda)=\sum_{i=1}^{n^0_{\lambda\lambda}}|\chi^i_{\lambda\lambda}(e)|^2>0,$
which is $W$-invariant with respect to $\lambda$.  Hence,  as in (\ref{hermitian toric}),  with respect to the  Hermitian metric $h$ given by (\ref{Fubini-Study potential}), we get from (\ref{Bergeman-coordinate}),
\begin{eqnarray}\label{Bergeman toric}
\hat b_{L^m,h^m}|_{T^\mathbb C}(x)=\frac{\sum_{\lambda\in mP\cap\mathfrak M}n'(\lambda) e^{2\lambda (x)}}{\left(\sum_{\lambda\in m_0P\cap\mathfrak M}\bar n(\lambda)e^{2\lambda(x)}\right)^{\frac {m}{m_0}}}
\end{eqnarray}
where $n'(\lambda)=\sum_{i=1}^{n_{\lambda\lambda}}|\chi^i_{\lambda\lambda}(e)|^2>0,$
which is also $W$-invariant with respect to $\lambda$.

Next we estimate the function $b^\Pi_{L^m,h^m}(\cdot)$ defined by  \eqref{partial Bergeman kernel}. Recall that $G=K^\mathbb C$. By the highest weight classification theorem (cf. \cite[Theorem 7.34]{Sep}) and the decomposition \eqref{H^0}, any
$$\Pi\in Gr_k^{K\times K}(H^0(M,L^m))$$
can be decomposed as
\begin{eqnarray}\label{Pi decomp}
\Pi=\oplus_{\lambda\in I_\Pi}\text{End}(V_\lambda)
\end{eqnarray}
for some set of dominant weights $I_\Pi\subseteq mP_+\cap \mathfrak M$.Let
\begin{eqnarray}\label{hat I Pi}
\hat I_\Pi =\text{Conv}(\{w(\lambda)|~\lambda\in I_\Pi,~w\in W\})
\end{eqnarray}
be the convex hull of $\cup_{w\in W}w(I_\Pi)$.
Then analogously to \eqref{Bergeman toric}, by \eqref{partial Bergeman kernel} we get
\begin{eqnarray}\label{Bergeman pi pinch}
b^\Pi_{L^m,h^m}|_{T^\mathbb C}(x)\simeq\frac1{h^m(x)}\sum_{\lambda\in I_\Pi}\sum_{s^{i}_{\lambda_1\lambda_2}\in\text{End}(V_\lambda)}|\chi^{i}_{\lambda_1\lambda_2}(x)|^2=:\hat b^\Pi_{L^m,h^m}(x)
\end{eqnarray}
for a $K\times K$-invariant function $\hat b^\Pi_{L^m,h^m}(x)$, and
\begin{eqnarray}\label{Bergeman pi toric}
\hat b^\Pi_{L^m,h^m}|_{T^\mathbb C}(x)=\frac{\sum_{\mu\in \hat I_\Pi \cap \mathfrak M }n''(\mu) e^{2\mu (x)}}{\left(\sum_{\lambda\in m_0P\cap\mathfrak M}\bar n(\lambda)e^{2\lambda(x)}\right)^{\frac {m}{m_0}}}
\end{eqnarray}
for some $W$-invariant $n''(\lambda)\in\mathbb R_{\geq0}$.  We note that $n''(\lambda)\in\mathbb R_{+}$ for any $\lambda\in\cup_{w\in W} w(I_\Pi)$.

\section{Computation of $\alpha_{m,k}^{K\times K}(M,L^m)$}\label{sect computation}

Let    $\{F_A\}_{A=1}^{d_0}$  be all   facets of  polytope  $P$ associated  to  $(Z, L|_Z)$.  Then there are  defining prime inner  norms $u_A$, namely the indivisible inner normal vectors in $\mathfrak N=\text{Hom}_\mathbb Z(\mathfrak M,\mathbb Z)$, such that
$$F_A \subseteq\{y\in\mathfrak a^*|~ l_A(y) =0\}$$
and
\begin{align}\label{linear-P}P=\cap_{A=1}^{d_0}\{y\in\mathfrak a^*|~l_A(y)\ge 0\},
\end{align}
where $l_A(y)=\Lambda_A+ u_A(y)$  are affine functions for some constants $\Lambda_A$.
Let $r$ be the dimension of $T$. Since  $P$ satisfies the Delzant condition (cf. \cite[Section 2]{Ab}), for each vertex $p$ of $P$, there are precisely $r$ facets $\{F_a\}_{a=i_1}^{i_r}\subseteq\{F_1,...,F_{d_0}\}$ meeting at $p$. We denote a cone generated by the  prime inner norms  $\{u_a\}_{a=i_1}^{i_r}$   of  $\{F_a\}_{a=i_1}^{i_r}$ by
\begin{eqnarray}\label{sigma p}
\sigma_p=\text{Span}_{\mathbb R_+}\{u_a|  ~a=i_1,...,i_r\}.
\end{eqnarray}

 Let $\mathcal L$ be a toric line bundle on $Z$   given by
 $$\mathcal L=\sum_A \Lambda_A' D_A$$
 for some constants $\Lambda_A'$, where $D_A$ is the prime toric divisor associated to $u_A$.
Then for each $\sigma_p$, there is a $\nu_p\in \mathfrak a^*$, which is the unique solution of linear system
\begin{eqnarray}\label{coef-Ups}
\nu_p(u_a)=-\Lambda_a',a=i_1,...,i_r.
\end{eqnarray}
Thus there is a unique  piecewise linear function $\Upsilon_\mathcal L(\cdot)$ defined on $\mathfrak a$ such that on each cone $\sigma_p$,  it  is defined by (cf. \cite[Definition 4.2.11]{CLS}),
\begin{align}\label{L-function}
\Upsilon_\mathcal L|_{\sigma_p}(x)=\nu_p(x).
\end{align}

Let $(M,L)$ be any polarized compactification of $G$ and $\omega=\sqrt{-1}\partial\bar\partial u\in 2\pi c_1(L)$  be a $K\times K$-invariant K\"ahler metric on $G$. Then by \eqref{0222-01} $u$ induces an associated function on $\mathfrak a$ (still denoted by $u$). Denote by $\nabla u,\nabla^2u$ the derivatives of  associated function $u$. We prove

\begin{lem}\label{determinent}
%There are constants $C,c>0$ such that
It holds the following uniform estimate
\begin{eqnarray}\label{asymp vol}
%c\leq\frac{\det(\nabla^2u)\prod_{\alpha\in\Phi_+}\langle\alpha,\nabla u\rangle^2(x)}{ e^{2\Upsilon_{K^{-1}_M|_Z}(-x)}\mathbf J(x)}\leq C,~\forall x\in\mathfrak a.
{\det(\nabla^2u)\prod_{\alpha\in\Phi_+}\langle\alpha,\nabla u\rangle^2(x)}\simeq{ e^{2\Upsilon_{K^{-1}_M|_Z}(-x)}\mathbf J(x)},~\forall x\in\mathfrak a.
\end{eqnarray}
\end{lem}

\begin{proof}
%Recall the trivialization \eqref{trivialization}.
By the volume form $\omega^n$ we can define a smooth Hermitian metric $h_1$ of $K^{-1}_M$.  Thus $h_1|_Z$ gives a smooth Hermitian metric of $K^{-1}_M|_Z$. Choose a local trivialization $$s_o=dV_G$$ of $K^{-1}_M$ on $G$, where $dV_G$ is  the Haar measure of $G$ as in Proposition \ref{KAK int}. Then on $G$,
%under the trivialization \eqref{trivialization},
$$\omega^n|_G=\det(\sqrt{-1}\partial\bar\partial u)dV_G.$$
It follows that
$$h_1|_{G}=\det(\sqrt{-1}\partial\bar\partial u)|_{G}.$$
Under the local trivialization $s_o|_Z$, by a formula in \cite[Corollary 1.3]{Del2},  we  see that at each $x\in \mathfrak a$,
\begin{align}\label{volume-on-Z}
h_1(\exp(x))&=\det(\sqrt{-1}\partial\bar\partial u)(\exp(x))\notag\\&=\frac{\det(\nabla^2u)\prod_{\alpha\in\Phi_+}\alpha^2(\nabla u)(x)}{2^{n+r}\mathbf J(x)}.
\end{align}
On the other hand, for any toric line bundle $\mathcal L$,  there is a Batyrev-Tschinkel  metric  $h_{BT}$  on $\mathcal L$,  which is  a $T$-invariant continuous Hermitian metric   \cite[Section 4]{Batyrev-Tschinkel}.
% constructed a , which is usually refereed as Batyrev-Tschinkel .
 In case  of  $\mathcal L=K^{-1}_M|_Z$,  the metric $h_{BT}$ is determined by
$$h_{BT}|_{T^\mathbb C}=e^{2\Upsilon_{K^{-1}_M|_Z}(-x)},~\forall x\in\mathfrak a,$$
with respect to $s_o$  (cf.  \cite[Proposition-Definition 3.3.1]{Maillot}).  Note that $\frac{h_1}{h_{BT}}$ is a globally defined positive function on $Z$.   Hence we get
\begin{align}\label{ricci-potential-1}\sup_{\mathfrak a}\left|\log\frac{\det(\nabla^2u)\prod_{\alpha\in\Phi_+}\alpha^2(\nabla u)(x)}{2^{n+r}\mathbf J(x)}-2\Upsilon_{K^{-1}_M|_Z}(-x)\right|<+\infty.
\end{align}
The lemma is proved.
\end{proof}

\begin{rem}In the proof of Lemma \ref{determinent}, we used the  Batyrev-Tschinkel  metric  $h_{BT}$  on $K^{-1}_M|_Z$ introduced in  \cite{Batyrev-Tschinkel}. In  Appendix  we  will  give a direct proof  of (\ref{ricci-potential-1}).
\end{rem}

For any convex set $\Omega\subseteq\mathfrak a^*$, we define the support function associated to $\Omega$  by $v_\Omega(\cdot)$,
\begin{align}\label{support-function}v_\Omega(x)=\sup_{y\in\Omega}\langle x,y\rangle,\forall x\in\mathfrak a.
\end{align}
Let $v_P(\cdot)$ and  $v_{\hat I_\Pi}(\cdot)$ be the support functions of  $P$ and the convex hull $\hat  I_\Pi$  of $I_\Pi$,
%as in (\ref{support-function})
respectively.  Then by Lemma \ref{determinent},  we are able  to compute the  $\alpha^{K\times K}_{m,k}$-invariant in terms of $v_P(\cdot)$ and  $v_{\hat I_\Pi}(\cdot)$.
%as in  \cite{SongAJM}, where
%$\alpha^{\hat K}(M)$-invariant  is studied on  toric manifolds.
%We will drop the assumption that $M$ is Fano in \cite{Del3}.% Also, since there is a new term $b^\Pi_{L^m,h^m}$, the Bergman kernel of a $k$-plane $\Pi$ to be considered, we include the main arguments for completeness.

\begin{theo}\label{group alpaha m-k thm general}
Let $(M,L)$ be any polarized compactification of $G$ and $\Pi\in Gr_k^{K\times K}(H^0(M,L^m))$. Then
\begin{eqnarray}\label{group alpha pi m-k general}
\alpha^\Pi_{m,k}=\sup\left\{\alpha\in(0,1)\left|\alpha v_P(x)+\Upsilon_{{K^{-1}_M}|_Z}(-x)+2\rho(x)-\frac{\alpha}m v_{\hat I_\Pi}(x)<0\text{ on }{\mathfrak a_+}\right.\right\},
\end{eqnarray}
where
$\rho=\frac12\sum_{\alpha\in\Phi_+}\alpha$ is half the sum of all positive roots in $\Phi_+$. Moreover,
\begin{eqnarray}\label{group alpha m-k general}
\alpha^{K\times K}_{m,k}(M,L)=\min\left\{\alpha^\Pi_{m,k}|~\Pi\in Gr_k^{K\times K}(H^0(M,L^m))\right\}.
\end{eqnarray}
\end{theo}

\begin{proof}Let $h$ be the Hermitian metric on $L$  given  by  \eqref{Fubini-Study potential}. Then by \eqref{hermitian toric},
$$
h|_\mathfrak a(x)=\left(\sum_{\lambda\in m_0P\cap\mathfrak M}\bar n(\lambda)e^{2\lambda(x)}\right)^{-\frac1{m_0}}.
$$
Set $u(x)=-\log h(x)$, we have
\begin{eqnarray}\label{potential toric}
u|_\mathfrak a(x)=\frac1{m_0}\log \left(\sum_{\lambda\in m_0P\cap\mathfrak M}\bar n(\lambda)e^{2\lambda(x)}\right),
\end{eqnarray}
and $\omega_h={\sqrt{-1}}\partial\bar\partial u$ gives a $K\times K$-invariant metric in $2\pi c_1(L)$ on $M$.

By (\ref{Bergeman pi toric}), we have
$$\hat b^\Pi_{L^m,h^m}|_Z(x) \simeq \frac{\sum_{\mu\in \hat I_\Pi} e^{2\mu(x)}}{\left(\sum_{\lambda\in m_0P\cap\mathfrak M}e^{2\lambda(x)}\right)^{\frac m{m_0}}}.$$
Hence,  by \eqref{Bergeman pi pinch}, Proposition \ref{KAK int} and Lemma \ref{determinent}, it follows that
\begin{eqnarray}\label{0201}
\begin{aligned}
\int_M(b^\Pi_{L^m,h^m})^{-\frac\alpha m}\omega_h^n&\simeq
\int_M(\hat b^\Pi_{L^m,h^m})^{-\frac\alpha m}\omega_h^n\\
&\simeq\int_{\mathfrak a_+}\frac{\det(\nabla^2u)\prod_{\alpha\in\Phi_+}\langle\alpha,\nabla u\rangle^2}{(\hat b^\Pi_{L^m,h^m})^{\frac\alpha m}}dx\\
&\simeq
\int_{\mathfrak a_+}
{\frac{\left(\sum_{\lambda\in m_0P\cap\mathfrak M}e^{2\lambda(x)}\right)^{\frac\alpha {m_0}}e^{2\Upsilon_{K^{-1}_M|_Z}(-x)}\mathbf J(x)}{\left(\sum_{\mu\in \cap \hat I_\Pi}e^{2\mu(x)}\right)^{\frac{\alpha} m}}}dx.
\end{aligned}
\end{eqnarray}

For any $N$ points $\Lambda=\{\lambda_1,...,\lambda_N\}\subseteq\mathfrak a^*$ it holds
$$\sum_{i=1}^Ne^{\lambda_i(x)}\simeq e^{v_{\hat\Lambda}(x)},|x|\to\infty,$$
where $v_{\hat\Lambda}(\cdot)$ is the support function of  convex hull $\hat \Lambda$ of $\Lambda$ as defined in  (\ref{support-function}).
Then
\begin{align}
\left(\sum_{\lambda\in m_0P\cap\mathfrak M}e^{2\lambda(x)}\right)^{\frac\alpha {m_0}}\simeq e^{\frac\alpha {m_0}v_{m_0P}(x)}
=e^{\alpha v_{P}(x)},\notag
\end{align}
and
\begin{align}
\sum_{\mu\in \cap \hat I_\Pi}e^{2\mu(x)}\simeq v_{\hat I_\Pi(x)}.\notag
\end{align}
Note that
$\mathbf J(x)\simeq e^{4\rho(x)},\text{ as }|x|\to\infty\text{ in }\mathfrak a_+.$
Thus  by \eqref{potential toric} and \eqref{0201}, we obtain
\begin{eqnarray*}
\begin{aligned}
\int_M(b^\Pi_{L^m,h^m})^{-\frac\alpha m}\omega_h^n
&\simeq\int_{\mathfrak a_+}
e^{2\left(\alpha v_{P}(x)+\Upsilon_{K^{-1}_M|_Z}(-x)+2\rho(x)-{\frac{\alpha}{m}}v_{\hat I_\Pi}(x)\right)}dx.
\end{aligned}
\end{eqnarray*}
Since
$${\alpha}v_{P}(x)+\Upsilon_{K^{-1}_M|_Z}(-x)+2\rho(x)-{\frac{\alpha}{m}}v_{\hat I_\Pi}(x)$$
is a piecewise linear function whose linear domains are cones, we may divide $\mathfrak a_+$ into several convex cones with common vertex $O$ such that it is linear on each cone. Hence,
$$\int_M(b^\Pi_{L^m,h^m})^{-\frac\alpha m}\omega_h^n<+\infty$$
if and only if
\begin{eqnarray*}
{\alpha}v_{P}(x)+\Upsilon_{K^{-1}_M|_Z}(-x)+2\rho(x)-{\frac{\alpha}{m}}v_{\hat I_\Pi}(x)<0,~\forall x\in\overline{\mathfrak a_+}\setminus\{O\}.
\end{eqnarray*}
Therefore,  we prove \eqref{group alpha pi m-k general}.

It remains to show \eqref{group alpha m-k general}. By \eqref{Pi decomp}, the number of $K\times K$-invariant $k$-dimensional subspaces of $H^0(M,L^m)$ is finite. In fact,
 $$\#Gr_k^{K\times K}(H^0(M,L^m))\le  2^{\#(mP_+\cap\mathfrak M)}.$$
Thus by  (\ref{alpha m-k-K def}), we also  get \eqref{group alpha m-k general}.

\end{proof}

Applying  Theorem \ref{group alpaha m-k thm general} to Fano $G$-group compactifications,  we   have the following corollary.

\begin{cor}\label{group alpaha m-k thm} Let $M$ be a Fano $G$-group compactification and  $P$ the polytope associated to $(M,K^{-1}_M)$.
Then for any $m,k$, it holds  for any  $\Pi\in Gr_k^{K\times K}(H^0(M,K^{-m}_M))$,
\begin{eqnarray}\label{group alpha pi m-k+}
\alpha^\Pi_{m,k}=\sup\left\{\alpha\in(0,1)~ \left|\left(mP+{\frac{\alpha}{1-\alpha}}\hat I_\Pi-{\frac{2m}{1-\alpha}}\rho\right)\cap\overline{\mathfrak a_+^\vee}\not=\emptyset\right.\right\},
\end{eqnarray}
where $a_+^\vee$ is the linear  dual cone of the positive Weyl chamber $a_+$.
Moreover,
\begin{eqnarray}\label{group alpha m-k+}
\alpha^{K\times K}_{m,k}(M,K_M^{-1})=\min \left\{\alpha^\Pi_{m,k}|~\Pi\in Gr_k^{K\times K}(H^0(M,K^{-m}_M))\right\}.
\end{eqnarray}
\end{cor}

\begin{proof}
When $M$ is Fano and $L=K^{-1}_M$, we have (cf. \cite[Proposition 4.2.14]{CLS})
\begin{align}\label{Upsilon reduce}
\Upsilon_{K^{-1}_M|_Z}(-x)=-v_{P}(x).
\end{align}
Then \eqref{group alpha pi m-k general} can be reduced to
\begin{eqnarray}\label{eq-0302}
mv_{P}(x)+{\frac{\alpha}{1-{\alpha}}}v_{\hat I_\Pi}(x)-\frac{2m}{1-{\alpha}}\rho(x)>0,\forall x\in\mathfrak a_+.
\end{eqnarray}
 It remains  to check that for any $\alpha$ satisfying (\ref{eq-0302}) it holds
\begin{eqnarray}\label{eq-0301}
\left(mP+{\frac{\alpha}{1-\alpha}}\hat I_\Pi-{\frac{2m}{1-\alpha}}\rho\right)\cap{\mathfrak a_+^\vee}\not=\emptyset.
\end{eqnarray}

Suppose that \eqref{eq-0301} is not true.  Note that  $\left(mP+{\frac{\alpha}{1-\alpha}}\hat I_\Pi-{\frac{2m}{1-\alpha}}\rho\right)$ and $\mathfrak a_+^\vee$ are both convex.  Then there is a $u\in\mathfrak a$ such that
\begin{eqnarray}\label{sign-on-P}\langle u,y\rangle<0,~\forall y\in \mathcal P=\left(mP+{\frac{\alpha}{1-\alpha}}\hat I_\Pi-{\frac{2m}{1-\alpha}}\rho\right)\end{eqnarray}
and
$$\langle u,x\rangle>0,~\forall x \in \mathfrak a_+^\vee.$$
The last inequality implies that $u\in{\mathfrak a}_+$. By \eqref{sign-on-P},
$$\sup_{y\in\mathcal P}\langle u,y\rangle \leq 0.$$
Thus by taking $x=u$ in \eqref{eq-0302}, we get a contradiction.  Hence, \eqref{eq-0301} is true and
 \eqref{group alpha pi m-k+} is proved.   \eqref{group alpha m-k+}  is in fact (\ref{group alpha m-k general}).
\end{proof}

 \subsection {$\alpha^{K\times K}(M,L)$-invariant }

  In this subsection, we  apply Theorem \ref{group alpaha m-k thm general} in case of $k=1$ to  compute  $\alpha^{K\times K}(M,L)$-invariant.

\begin{cor}\label{alpha-KtimesK}
Let $(M,L)$ be a polarized compactification of $G$. Then  $\alpha(M, L)$-invariant with respect to the $K\times K$-action can be computed by
\begin{eqnarray}\label{alpha K-K general case}
\begin{aligned}
&\alpha^{K\times K}(M,L)\\
&= \sup\left\{\alpha\in(0,1)\left|~\alpha v_P(x)+\Upsilon_{{K^{-1}_M}|_Z}(-x)+2\rho(x)-\alpha v_z(x)<0\text{ on }{\mathfrak a_+},\forall v_z\in\mathfrak a_z\right.\right\}.
\end{aligned}
\end{eqnarray}
Here $\mathfrak a_z$ is the centre of Lie algebra of $G$.
\end{cor}

\begin{proof}
In case of  $k=1$,   $I_\Pi=\hat I_\Pi =\{mv_z\}$ for some $v_z\in\mathfrak a^*_z\cap\frac1m\mathfrak M$.  Thus by (\ref{alpha K}),  we get \eqref{alpha K-K general case} from \eqref{group alpha pi m-k general}  for $\alpha_{m,1}^{K\times K}(M,L)$ by taking  $m\to+\infty$.

\end{proof}

We remark that under the assumption that $M$ is Fano, Delcroix   \cite{Del1}  showed in a different way that
\begin{eqnarray}\label{Del res}
\alpha^{K\times K}(M,L)=\sup\{t\left|~t(P(M,L)-(P(M,L)\cap\mathfrak a_z^*))\subseteq P(M,K^{-1}_M)\ominus H \right.\},
\end{eqnarray}
where $P(M,L)$ is the polytope of $(M,L)$, $H=\text{Conv}(\{w(2\rho)|~w\in W\})$ and $$Q\ominus H:=\{x\in Q|~x+H\subseteq Q\}.$$
It can be shown that (\ref{Del res}) is equivalent to  (\ref{alpha K-K general case}).
In fact,   by  \eqref{Upsilon reduce},  (\ref{Del res})
is equivalent to
$$\alpha^{K\times K}(M,L)=\sup\left\{\alpha\in(0,1)\left|\left(-\alpha P(M,L)+P(M,{K^{-1}_M})-2\rho+\alpha v_z\right)\cap{\mathfrak a_+^\vee}\not=\emptyset,\forall v_z\in\mathfrak a_z^*\right.\right\}.$$
On the other hand, we have (cf. \cite[Section 3.2]{LZZ} or \eqref{uA-rhoA}-(\ref{anti-cannonical}) below),
 $$P(M,K^{-1}_M)=\cap_A\{y|1+\sum_{\alpha\in\Phi_+}|\langle\alpha,u_A\rangle|+\langle u_A,y\rangle\geq0\}.$$
Thus combining with  (\ref{linear-P}), we get
$$P(M,K^{-1}_M)\ominus H=\cap_A\{y|~1+\langle u_A,y\rangle\geq0\}.$$
Hence  \eqref{alpha K-K general case} is equivalent  to  \eqref{Del res}.

In case of Fano compactifications,   by \eqref{Upsilon reduce},   Corollary  \ref{alpha-KtimesK} can be restated as follows.

\begin{cor}\label{group alpha thm} Let $M$ be a Fano compactification of $G$. Then  $\alpha(M)$-invariant with respect to the $K\times K$-action is given by
\begin{eqnarray}\label{alpha K-K}
\alpha^{K\times K}(M)=\inf_{v_z\in P\cap\mathfrak a_z^*}\sup\left\{t\in(0,1)\left|{\frac{2}{1-t}}\rho-\frac{t}{1-t}v_z\in P_+\right.\right\}.
\end{eqnarray}
\end{cor}

\begin{rem}\label{rem37}
When $G$ is semisimple,  $\mathfrak a_z=\{O\}$.  Then \eqref{alpha K-K general case}  becomes
\begin{align}
&\alpha^{K\times K}(M,L)%\notag\\&
= \sup\left\{ \alpha\in(0,1)|~\alpha v_P(x)+\Upsilon_{{K^{-1}_M}|_Z}(-x)+2\rho(x)<0 ~{\rm  on }~{\mathfrak a_+}  \right\}\notag.
\end{align}
Assuming in addition that $M$ is Fano,   by (\ref{alpha K-K}),   we further get
\begin{eqnarray*}
\alpha^{K\times K}(M)=\sup\left\{t\in(0,1)\left|{\frac{2}{1-t}}\rho\in P_+\right.\right\}.
\end{eqnarray*}
\end{rem}

\subsection{Proof of  Theorem \ref{alpha m-1 K-K prop}}

In this subsection,  we prove  Theorem \ref{alpha m-1 K-K prop}.  First  we use Corollary  \ref{alpha-KtimesK}  to  give an explicit computation  of $\alpha^{K\times K}(M,L)$ in terms of linear functions $l_A$.

\begin{prop}\label{group alpha del thm}
Let $(M,L)$ be the polarized compactification of $G$ and $l_A$ the affine functions given in (\ref{linear-P}). Then
\begin{eqnarray}\label{all case}
\alpha^{K\times K}(M,L) =\min_{A=1,...,d_0}\min_{v_z\in\mathfrak a^*_z\cap  P}\frac{1}{l_{A} (v_z)}.
\end{eqnarray}
\end{prop}

\begin{proof}
 Let   ${\rm Vert}(P)$  be the set of  vertices of $P$.   For each $p\in$Vert$(p)$,  let $\sigma_p$  be  the cone defined as in \eqref{sigma p}.  Then $$\cup_{p\in\text{Vert}(P)}(-\sigma_p)=\mathfrak a.$$
Thus by  \eqref{alpha K-K general case},  we get
\begin{align}\label{alpha on cone}
\alpha^{K\times K}(M,L)
=\min_{p\in\text{Vert}(P)}\sup_{\alpha\in(0,1)}   \{ &\alpha v_P(x)+\Upsilon_{{K^{-1}_M}|_Z}(-x)+2\rho(x)-\alpha v_z(x)<0  \notag\\
&\text{ on }{(-\sigma_p)}\cap\mathfrak a_+,\forall v_z\in\mathfrak a_z \}.
\end{align}
For each cone $\sigma_p$,  we set
\begin{eqnarray}\label{tsigma}
\tau_\sigma=  \sup_{t\in(0,1)}   \{ tv_P(x)+\Upsilon_{{K^{-1}_M}|_Z}(-x)+2\rho(x)-tv_z(x)<0\text{ on }{(-\sigma_p)}\cap\mathfrak a_+,\forall v_z\in\mathfrak a_z\}. \end{eqnarray}
We need to estimate each $\tau_\sigma$.
%To simplify the expression, note that the piecewise linear function $\Upsilon |_{K^{-1}_M}(-x)$ is indeed linear on each $\sigma_p$, by definition (see the beginning of Section 3).

First we want to simplify the term $\Upsilon_{{K^{-1}_M}|_Z}(-x)$ when $x\in(-\sigma_p)$.  Note that  for each defining prime inner  norm $u_A$,  $A\in\{1,...,d_0\}$,  there is an element $w\in W$ such that $w(u_A)\in(-\overline{\mathfrak a^*_+})$.  Then  we  can define  a point $\rho_A\in\mathfrak a^*$  by
$$\rho_A=w^{-1}(\rho).$$
One can check that
\begin{align}\label{uA-rhoA}
\rho_A(u_A)=-\frac12\sum_{\alpha\in\Phi_+}|\alpha(u_A)|
\end{align}  is well defined for each $A$.
Now  the divisor $-K_M|_Z$ can be  written  as sum of prime toric divisors of $Z$.  Namely, we have
\begin{align}\label{anti-cannonical}
-K_M|_Z=\sum_A(1-2\rho_A(u_A))D_A.
\end{align}

Let us sketch the proof of \eqref{anti-cannonical} following \cite[Section 3.2.4]{Del3}. Denote by $B^+$ the (positive) Borel subgroup of $G$ corresponding to $(T^\mathbb C,\Phi_+)$ and $B^-$ be the opposite one. Then by \cite[Section 3]{Brion89},  there exists a $B^+\times B^-$-semi-invariant section of $-K_M$,
$$-K_M=\sum_{A'}X_{A'}+2\sum_{\alpha_i\in\Phi_{+,s}}Y_{\alpha_i},$$
where $\{X_{A'}\}$ is the set of $G\times G$-invariant prime divisors and $Y_{\alpha_i}$ is the prime $B^+\times B^-$-semi-invariant divisor with weight $\alpha_i$ in $\Phi_{+,s}$, the set of simple roots in $\Phi_+$. Note that the corresponding $B^+\times B^-$-weight of this divisor is $2\rho$ (cf. \cite[Section 3.2.4]{Del3}). Thus by adding the divisor of a $B^+\times B^-$-semi-invariant rational function $f$ with weight $-2\rho$, we have
$$-K_M+\text{div}(f)$$
is a $G\times G$-invariant divisor. On the other hand, by \cite[Theorem 2.4 (3)]{AK}, the prime $G\times G$-invariant divisors of $M$ are in bijections with $W$-orbits of prime toric divisors of $Z$. Restricting the above divisor to $Z$, we get \eqref{anti-cannonical}.

By (\ref{anti-cannonical}),  we can compute $\Upsilon_{{K^{-1}_M}|_Z}(\cdot)$.  In fact, by  definition (see Section 3 above), $\Upsilon_{{K^{-1}_M}|_Z}(-x)$ is linear on $(-\sigma_p)$ and it satisfies that
\begin{eqnarray}
\Upsilon |_{K^{-1}_M}(-x)=-\hat p(x),
\end{eqnarray}
where $\hat p$ is determined by \eqref{coef-Ups} and  \eqref{anti-cannonical}. More precisely, let $\{F_{A_i}\}_{i=1}^r$ be the $r$ facets at $p$ and $\{u_{A_i}^{(p)}\}_{i=1}^r$ be the corresponding inner norms, then
%\footnote{One can also use \eqref{ricci-potential} and the arguments in \cite[Section 3.2]{LZZ} to get \eqref{upsilon-sigma-p}.}
\begin{eqnarray}\label{upsilon-sigma-p}
u_{A_i}^{(p)}(\hat p)=2\rho(u_{A_i}^{(p)})-1, ~i=1,...,r.
\end{eqnarray}

Next  we begin to  estimate  $\tau_\sigma $.   Since a cone $\sigma_p$ may intersect with  different Weyl chambers at the same time,
 we  will   divide into   two cases  in the following.

\emph{Case (1): $p\in\mathfrak a^*_+$.} In this case, $(-\sigma_p)\subseteq\overline{\mathfrak a_+}$. Then on $(-\sigma_p)=\text{Span}_{\mathbb R_+}\{-u_{A_i}^{(p)}\}_{i=1,...,r}$,  the support  function  $v_P(\cdot)$  satisfies
\begin{eqnarray}\label{vp-sigma-p}
v_P(x)=p(x).
\end{eqnarray}

Notice that  the left-hand side of \eqref{alpha K-K general case} is linear on $(-\sigma_p)$, which is spanned by $\{-u_{A_i}^{(p)}\}_{i=1,...,r}$.    Then to find $t$ satisfying \eqref{tsigma},  it suffices to require that
\begin{align}
t v_P(-u_{A_i}^{(p)})+\Upsilon_{K^{-1}_M|_Z}(u_{A_i}^{(p)})-2\rho(u_{A_i}^{(p)})+t u_{A_i}^{(p)}(v_z)<0,\forall v_z\in\mathfrak a_z\text{ and }i=1,...,r.\notag
\end{align}
Hence, by \eqref{upsilon-sigma-p} and \eqref{vp-sigma-p}, we get
\begin{eqnarray}\label{case-1}
\tau_\sigma&= \min_{i=1,...,r}\min_{v_z\in\mathfrak a^*_z\cap P}\frac{1}{\Lambda_{A_i}^{(p)}+ u_{A_i}^{(p)}(v_z)}\notag\\&=\min_{i=1,...,r}\min_{v_z\in\mathfrak a^*_z\cap P}\frac{1}{l_{A_i}^{(p)}(v_z)}.
\end{eqnarray}

\emph{Case (2): $p\in\cap_{j=1}^sW_{\alpha_s}$ for some Weyl walls $\{W_{\alpha_1},...,W_{\alpha_s}\}$.} In this case, $(-\sigma_p)$ intersects with some  different Weyl chambers. By the $W$-invariance of $P$, $(-\sigma_p)$ is invariant under the reflections $s_{\alpha_j}$ with respect to $W_{\alpha_j}$ for $j=1,...,s$. We need to determine a basis  of cone $(-\sigma_p)_+=(-\sigma_p)\cap\overline{\mathfrak a_+^*}$.

Let  $W_{\alpha_1},...,W_{\alpha_{s_+}}$ be a set of  facets of $\mathfrak a_+$ whose interior intersects $(-\sigma_p)$, where $\{1,...,s_+\}\subseteq\{1,...,s\}$. Let  $\{F_{A_i}^{(p)}\}_{i=1,...,d_+}$ be a set of  codimension $1$ facets of $P$ around $p$ such that $F_{A_i}^{(p)}\cap\overline{\mathfrak a_+^*}\not=\emptyset$.
 %and $u_{A_i}^{(p)}$ the defining prime inner norm of $F_{A_i}^{(p)}$.
 By convexity, we have
$$(-\sigma_p)_+=\text{Span}_{\mathbb R_+}\{(\cup_{i=1}^{d_+}u_{A_i}^{(p)})\cup(\cup_{j=1}^{s_+}(-\sigma_p)_+\cap W_{\alpha_j})\}.$$
 For each $j\in\{1,...,s_+\}$, we set
\begin{align*}(-\sigma_p)_+^j=(-\sigma_p)_+\cap W_{\alpha_j}.\end{align*}
Then
\begin{align*}(-\sigma_p)_+^j=\text{Span}_{\mathbb R_+}\{(\partial(-\sigma_p)\cap W_{\alpha_j}\cap\overline{\mathfrak a_+})\cup(\cup_{k\not=j}(-\sigma_p)^j_+\cap W_{\alpha_k})\}.
\end{align*}
Since $(-\sigma_p)$ is convex and $s_{\alpha_j}$-invariant, we have
\begin{align}\label{projection-1}
(\partial(-\sigma_P)\cap W_{\alpha_j}\cap\overline{\mathfrak a_+})\subseteq\text{Span}_{\mathbb R_+}\{\pi_{\alpha_j}(-u_{A_i}^{(p)})|~i=1,...,d_+\}\subseteq(-\sigma_p)_+^j,
\end{align}
where $\pi_{\alpha_j}$ is the projection to $W_{\alpha_j}$. It follows that
$$(-\sigma_p)_+^j=\text{Span}_{\mathbb R_+}\{\{\pi_{\alpha_j}(-u_{A_i}^{(p)})|~i=1,...,d_+\}\cup(\cup_{k\not=j}(-\sigma_p)^j_+\cap W_{\alpha_k})\}.
$$
Note that $(-\sigma_p)$ is invariant under the   reflections of $s_{\alpha_j}$ and $s_{\alpha_k}$. Thus similarly to (\ref{projection-1}), we also get
$$(\partial(-\sigma_P)\cap (W_{\alpha_j}\cap W_{\alpha_k})\cap\overline{\mathfrak a_+})\subseteq\text{Span}_{\mathbb R_+}\{\pi_{\alpha_j\alpha_k}(-u_{A_i}^{(p)})|~i=1,...,d_+\}\subseteq(-\sigma_p)_+^j,$$
where $\pi_{\alpha_j\alpha_k}$ is the projection to $W_{\alpha_j}\cap W_{\alpha_k}$.
Hence, by using an induction argument, we will obtain
%\begin{eqnarray}
\begin{align}\label{sigmap+-1}
&(-\sigma_p)_+\notag\\
%&=\text{Span}_{\mathbb R_+}\left(\{-u_{A_1}^{(p)},...,-u_{A_{d_+}}^{(p)}\}\cup\left(\cup_{j=1}^s\pi_{\alpha_j}(-\sigma_p)\right)\cap\overline{\mathfrak a_+^*}\right)\\
&=\text{Span}_{\mathbb R_+}\left\{\cup_{l=0}^{s_+}\{\pi_{\alpha_{j_1},...,\alpha_{j_l}}(-u_{A_i}^{(p)})|~i=1,...,d_+,\{j_1,...,j_l\}\subseteq\{1,...,s_+\}\}\right\},
\end{align}
%\end{eqnarray}
where $\pi_{\alpha_{j_1},...,\alpha_{j_l}}$ is the projection to $\cap_{k=1}^lW_{\alpha_{j_k}}$.

For simplicity, in the following we deal with the case of  $s=1$,  which means that $(-\sigma_p)$  intersects with only two Weyl chambers which share a common facet lying on a Weyl wall $W_\alpha$ for some $\alpha\in\Phi_+$. The general case can be done in a same way by using an induction argument.
Note that we have either $F_{A_i}^{(p)}\in\{F_{A_j}^{(p)}\}_{j=1,...,d_+}$ or $s_\alpha(F_{A_{i'}}^{(p)})\in\{F_{A_j}^{(p)}\}_{j=1,...,d_+}$. Thus, by (\ref{sigmap+-1}),  we get (See Figure 1)
\begin{eqnarray*}\label{sigmap+}
\begin{aligned}
(-\sigma_p)_+
=\text{Span}_{\mathbb R_+}\{-u_{A_i}^{(p)},\pi_{\alpha}(-u_{A_i}^{(p)})|i=1,...,d_+\}.
\end{aligned}
\end{eqnarray*}

\begin{figure}[h]
  \centering
  % Requires \usepackage{graphicx}
  \includegraphics[width=1.8in]{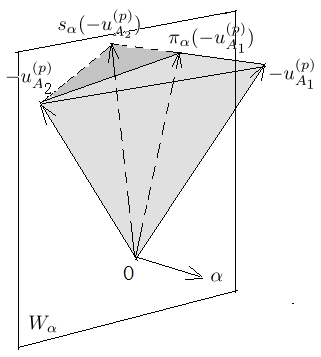}\\
  \caption{In this example, $(-\sigma_p)=\text{Span}_{\mathbb R_+}\{-u_{A_1}^{(p)},s_\alpha(-u_{A_2}^{(p)}),-u_{A_2}^{(p)}\}$, where $u_{A_1}^{(p)}\not\in W_\alpha$ and $u_{A_2}^{(p)}\in W_\alpha$. Then the positive part (marked by light grey) $(-\sigma_p)_+=\text{Span}_{\mathbb R_+}\{-u_{A_1}^{(p)},\pi_\alpha(-u_{A_1}^{(p)}),-u_{A_2}^{(p)}\}$. }%\label{}
\end{figure}

 By \eqref{coef-Ups},  we have
\begin{eqnarray*}\label{upsilon-sigma-p+}
\begin{aligned}
u_{A_i}^{(p)}(\hat p)&=2\rho(u_{A_i}^{(p)})-1,\text{ if }F_{A_i}^{(p)}\in\{F_{A_j}^{(p)}\}_{j=1,...,d_+},
\end{aligned}
\end{eqnarray*}
and
\begin{eqnarray*}
\begin{aligned}
u_{A_i}^{(p)}(\hat p)&=2s_\alpha(\rho)(u_{A_i}^{(p)})-1\\
&=2\rho(u_{A'_i}^{(p)})-1,\text{ if }F_{A'_i}^{(p)}=s_\alpha(F_{A_i}^{(p)})\in\{F_{A_j}^{(p)}\}_{j=1,...,d_+}.
\end{aligned}
\end{eqnarray*}
This implies that $\hat p\in W_\alpha$. Thus
\begin{eqnarray}\label{upsion-ua+}
\begin{aligned}
\Upsilon_{K^{-1}_M|_Z}(\pi_\alpha(u_{A_i}^{(p)}))&=\hat p(\pi_\alpha(u_{A_i}^{(p)}))\\
&=\hat p(u_{A_i}^{(p)}).
\end{aligned}
\end{eqnarray}
Similarly, by the fact that $p\in W_\alpha$, we get
\begin{eqnarray}\label{vp-ua+}
\begin{aligned}
v_P(\pi_\alpha(u_{A_i}^{(p)}))&= p(\pi_\alpha(u_{A_i}^{(p)}))\\
&= p(u_{A_i}^{(p)}).
\end{aligned}
\end{eqnarray}

Since $P$ is convex and $s_\alpha$-invariant,  for  an inner norm  $u_{A_i}^{(p)}$ of a facet in $\mathfrak a_+^*$,  we have
$$\alpha(u_{A_i}^{(p)})<0.$$
Recall that
$$\pi_\alpha(x)=x-2\frac{\langle\alpha,x\rangle}{|\alpha|^2}\alpha.$$
It follows that
\begin{eqnarray}\label{2rho-ua+}
\begin{aligned}
-2\rho(\pi_\alpha(u_{A_i}^{(p)}))&= -2\rho(u_{A_i}^{(p)})+\frac{2\alpha(u_{A_i}^{(p)})}{|\alpha|^2}\langle\alpha,\rho\rangle\\
&< -2\rho(u_{A_i}^{(p)}),~i\in\{1,...,d_+\}.
\end{aligned}
\end{eqnarray}
On the other hand, since $v_z\in\mathfrak a_z$,
\begin{eqnarray}\label{vz-ua+}
\begin{aligned}
v_z(\pi_\alpha(u_{A_i}^{(p)}))=v_z(u_{A_i}^{(p)}).
\end{aligned}
\end{eqnarray}
Hence,  combining \eqref{upsion-ua+}-\eqref{vz-ua+}, we derive
%\begin{eqnarray*}
\begin{align}\label{similar-2}
&\alpha v_P(-\pi_\alpha(u_{A_i}^{(p)}))+\Upsilon_{K^{-1}_M|_Z}(\pi_\alpha(u_{A_i}^{(p)}))-2\rho(\pi_\alpha(u_{A_i}^{(p)}))+\alpha v_z (\pi_\alpha(u_{A_i}^{(p)}))\notag\\
&\leq\alpha v_P(-u_{A_i}^{(p)})+\Upsilon_{K^{-1}_M|_Z}(u_{A_i}^{(p)})-2\rho(u_{A_i}^{(p)})+\alpha v_z (u_{A_i}^{(p)}),
\end{align}
%\end{eqnarray*}
for $i=1,...,d_+$ and any $v_x\in\mathfrak a_z$.

By (\ref{similar-2}),  we can  repeat  the argument in the proof of  \eqref{case-1} in \emph{Case (1)} to get
\begin{eqnarray*}
\tau_\sigma=\min_{i=1,...,d_+}\min_{v_z\in\mathfrak a^*_z\cap P}\frac{1}{l_{A_i}^{(p)}(v_z)}.
\end{eqnarray*}
Since both of $P$ and $v_z$ are $W$-invariant, we obtain
\begin{eqnarray}\label{case-2}
\tau_\sigma=\min_{i=1,...,r}\min_{v_z\in\mathfrak a^*_z\cap  P}\frac{1}{l_{A_i}^{(p)}(v_z)}.
\end{eqnarray}
Hence, combining  \eqref{alpha on cone} and \eqref{case-1} (or (\ref{case-2})), we finally  obtain \eqref{all case}. The proposition is proved.
\end{proof}

\begin{rem}\label{alpha-1}  \eqref{all case} has    been obtained  by Cheltsov-Shramov  for toric Fano varieties \cite[Section 5]{Cheltsov-Shramov2008} and by Blum-Jonsson  for  general polarized toric manifolds \cite[Section 7]{Blum-Jonsson17} in a different way. They  both proved \eqref{all case}  by computing log canonical thresholds as in \cite{DemaillyKollar}.
\end{rem}

Next we prove Theorem \ref{alpha m-1 K-K prop} by using Proposition \ref{group alpha del thm}.

\begin{proof}[Proof of Theorem \ref{alpha m-1 K-K prop}]
 For any fixed $A\in\{1,...,d_0\}$, the maximum  of the linear function
$l_A(v_z)$ must be attained by some vertex $v_{z_A}$ of $ P\cap\mathfrak a^*_z$. Then by \eqref{all case}, there is some $A_0\in\{1,...,d_0\}$ and vertex $v_{z_0}$ of $ P\cap\mathfrak a^*_z$ such that
\begin{align}\label{explicity-alpha-KK}
 \alpha^{K\times K}(M,L)=\min_{A=1,...,d_0}\min_{v_z\in\mathfrak a^*_z\cap  P}\frac{1}{l_{A} (v_z)}=\frac1{l_{A_0}(v_{z_0})}.
\end{align}

Note that the vertex $v_{z_0}$ is determined by integral equations
$$\alpha(v_{z_0})=0,~\forall\alpha\in\Phi_+,$$
and
$$l_{A_i}(v_{z_0})=0,\text{ for }i\in\{s_1,...,s_q\}\subseteq\{1,...,d_0\},$$
where $\{s_1,...,s_q\}$ denotes the indexes such that $v_{z_0}\in F_{A_{s_l}}$ for $l=1,...,q$.
Thus $v_{z_0}$ is a rational point in $\mathfrak M_{\mathbb Q}\cap\mathfrak a_z^*$ and there exists  the smallest integer $m_0\in\mathbb N_+$ such that $m_0v_{z_0}\in\mathfrak M\cap\mathfrak a_z^*$.

  For each $l\in\mathbb N_+$, the set of single point,
$$I_\Pi=\{lm_0v_{z_0}\}\subseteq (lm_0P)\cap\mathfrak M$$
  determines a $K\times K$-invariant   subspace of  dimension one,
  $$\Pi\in Gr^{K\times K}_1(H^0(M,L^{lm_0})).$$
  We claim:
  \begin{align}\label{claim-1}\alpha^{\Pi}_{m_0 l,1} =\frac1{l_{A_0}(v_{z_0})},~\forall~ l\in\mathbb N_+.
  \end{align}
  In fact,    by \eqref{group alpha pi m-k general}, for any  $l\in\mathbb N_+$,  we have
\begin{align}\label{relation-ml}
&\alpha^\Pi_{lm_0,1}\notag\\
&=\sup\left\{\alpha\in(0,1)\left|~\alpha v_P(x)+\Upsilon_{{K^{-1}_M}|_Z}(-x)+2\rho(x)- \alpha v_{z_0}(x)<0\text{ on }{\mathfrak a_+}\right.\right\}.
\end{align}
For each cone $(-\sigma_p)=\text{Span}_{\mathbb R_+}\{-u_{A_i}^{(p)}\}_{i=1,...,r}$, we set
$$
\tau^0_\sigma=\sup_{t\in(0,1)}   \{ v_P(x)+\Upsilon_{{K^{-1}_M}|_Z}(-x)+2\rho(x)-tv_{z_0}(x)<0\text{ on }{(-\sigma_p)}\cap\mathfrak a_+\}.$$
Then by following the argument in the proof of \eqref{case-1} and (\ref{case-2}) with  $v_z$ replaced by  $v_{z_0}$ in \eqref{tsigma},  we  get
\begin{align*}
\tau^0_\sigma
&=\min_{i=1,...,r}\frac{1}{l_{A_i}^{(p)}(v_{z_0})}\notag\\
&\geq\min_{A=1,...,d_0}\min_{v_z\in\mathfrak a^*_z\cap  P}\frac{1}{l_{A} (v_z)}=\frac1{l_{A_0}(v_{z_0})}.
\end{align*}
The last equality follows from (\ref{explicity-alpha-KK}).  Moreover,
$$ \tau^0_\sigma=\frac1{l_{A_0}(v_{z_0})}$$
if and only if $u_{A_0}\in\{u_{A_1}^{(p)},...,u_{A_r}^{(p)}\}$. Hence, by (\ref{relation-ml}),  we conclude that
\begin{align*}\alpha^{\Pi}_{m_0 l,1} =\frac1{l_{A_0}(v_{z_0})}.
\end{align*}
  (\ref{claim-1}) is proved.

 By (\ref{alpha K}) and  (\ref{claim-1}), we see that
 \begin{align}
 \alpha^{K\times K}(M,L)&=\lim_l \alpha^{K\times K}_{m_0 l,1}(M,L)\notag\\
 &\le \overline{\lim_l}\alpha^{\Pi}_{m_0 l,1} =\frac1{l_{A_0}(v_{z_0})}.
 \notag
 \end{align}
Combining with (\ref{explicity-alpha-KK}),  we obtain
 $$ \alpha^{K\times K}(M,L)=\alpha^{K\times K}_{m_0 l,1}(M,L),~\forall~ l\in\mathbb N_+.$$
Theorem \ref{alpha m-1 K-K prop} is proved.
\end{proof}

 Inspired by Theorem  \ref{birkar} and Theorem \ref{alpha m-1 K-K prop}, we propose the following quantization modification of Conjecture \ref{Tian-conj} for $\alpha_{m,1}^{\hat K}(M,L)$-invariant.

\begin{conj}\label{conj-modification}Let $(M,L)$ be a polarized  manifold and $\hat K$ a compact subgroup of  of  ${\rm Aut}(M)$. Suppose that $L$ is $\hat K$-linearized. %a maximal reductive subgroup ${\rm Aut}_r(M)$ of  ${\rm Aut}(M)$.
Let $H_{\hat K}^0(M, L^m)$  be the  maximal subspace of  $H^0(M, L^m)$ which is spanned by a basis of holomorphic sections in one-dimensional $\hat K$-invariant subspaces
%$\hat K$-invariant sections
of $H^0(M, L^m)$.  If in addition the ring $\oplus_{m\in \mathbb N} H_{\hat K}^0(M, L^m)$ is finitely  generated,  then
$$\alpha_{m_0l,1}^{\hat K}(M,L)=\alpha^{\hat K}(M,L),~\forall~ l\in \mathbb N_+.$$

\end{conj}

\section{ Toric Fano manifolds case}

In this section, we assume that $M$ is an $n$-dimensional toric Fano manifold and prove Theorem \ref{weak conj thm}, by using the results established in Section 3.  First we use  Corollary \ref{group alpaha m-k thm} to get the precise value of $\alpha_{m,k}^{T}(M)$ and $\alpha^{T}(M)$-invariants.

\begin{prop}\label{alpha m-k limit prop}
Let  $t(\cdot)$  be the function on $P$  defined  by (\ref{t(x)}). Then the followings are true:

(1) For any $T$-invariant $k$-subspace $\Pi\in Gr^{T}_k(H^0(M,L^m))$, we have
\begin{eqnarray}\label{alpha pi toric t}
\alpha_{m,k}^{\Pi}=\sup\{t(x)|~x\in \frac1m\hat I_\Pi\}
\end{eqnarray}
and
\begin{eqnarray}\label{alpha toric t}
\alpha^{T}_{m,k}(M,K_M^{-1})  =\min\left\{\sup_{x\in \frac1m\hat I_\Pi}t(x)|~I_\Pi\subseteq mP\cap\mathfrak M,\#I_\Pi=k\right\}.
\end{eqnarray}

(2)  \begin{eqnarray}\label{toric alpha}
\alpha^{T}(M)=\min_{x\in P}t(x).
\end{eqnarray}
Moreover,  for any fixed $k\in\mathbb N_+$, it holds
\begin{align}\label{alpha-Tn} \lim_{m\to\infty}\alpha^{T}_{m,k}(M,K_M^{-1})=\alpha^{T}(M).
\end{align}

\end{prop}

\begin{proof}
In toric case, $\rho=O$, $\mathfrak a_+=\mathfrak a$ and $\mathfrak a_+^\vee=\{O\}$.  By Corollary \ref{group alpaha m-k thm}, for any $T$-invariant $k$-subspace $\Pi\in Gr^{T}_k(H^0(M,L^m))$,
\begin{eqnarray*}
\alpha_{m,k}^{\Pi}=\sup\{\alpha\in\left(0,1\right)|~O\in mP+{\frac{\alpha}{1-\alpha}}\hat I_\Pi\}.
\end{eqnarray*}
This is equivalent to
$$\alpha_{m,k}^{\Pi}=\sup\{\alpha\in\left(0,1\right)|~\exists ~ x~\in\frac1m\hat I_\Pi\text{ such that }{\frac{\alpha}{\alpha-1}}x\in P\}.$$
Thus we get \eqref{alpha pi toric t}. \eqref{alpha toric t} then follows from \eqref{alpha pi toric t} and \eqref{alpha m-k-K def}.
Hence (1) is proved.

Next we prove (2).  Note that in case of $k=1$, in \eqref{alpha toric t} we have $\hat I_\Pi=I_\Pi$ is a single point in $mP\cap\mathfrak M$.  Then  by \eqref{alpha pi toric t},  we get
\begin{eqnarray*}
\alpha^{T}_{m,1}(M,K_M^{-1})  =\min\{t(x)|~x\in P\cap\frac1m\mathfrak M\}.
\end{eqnarray*}
Since $\cup_{m=1}^\infty (P\cap\frac1m\mathfrak M)$ is dense in $P$,  we derive   \eqref{toric alpha} by using  (\ref{alpha K}).

   By (\ref{toric alpha}), there exists a point  $x_0\in {P}$ such that $t(x_0)=\alpha^{T}(M)$.  It is easy to see that $t(\cdot)$ cannot attain its minimum in Int$(P)$. Thus for any $\epsilon>0$, there is a convex neighborhood $U_\epsilon$ of $x_0$ such that
\begin{align}\label{chioce-t} t|_{U_\epsilon}(x)\leq\alpha^{T}(M)+\epsilon.
\end{align}

For $m\in\mathbb N_+$ sufficiently large, we may assume $\#\left(U_\epsilon\cap\frac1m\mathfrak M\right)\geq k$.
Choose a $\Pi\in Gr^{\hat K}_k(H^0(M,L^m))$ such that
$$I_\Pi=\{m\lambda_0,...,m\lambda_{k-1}\}\subseteq \left(mU_\epsilon\cap\mathfrak M\right).$$
 Then $\frac1m\hat I_\Pi\subseteq U_\epsilon$.  Thus by  \eqref{alpha pi toric t} and (\ref{chioce-t}),  it follows that
$$ \alpha^\Pi_{m,k}<\alpha^{T}(M)+\epsilon.$$
 On the other hand,   by \eqref{toric alpha}, \eqref{alpha pi toric t} and \eqref{alpha m-k def}, we also have
\begin{align}\label{alpha-compare}\alpha^{T}(M)\leq\alpha^{T}_{m,k}(M,K_M^{-1})\leq\alpha^\Pi_{m,k}.\end{align}
Hence,  combining the above two inequalities, we  obtain  (\ref{alpha-Tn})  immediately
by taking $\epsilon\to 0$.
\end{proof}

\begin{rem}\label{rem1}
\eqref{toric alpha} coincides with a result of Song \cite{SongAJM}. In fact,   Song  considered  $\alpha^{\hat K}(M)$-invariant  for a larger group  $\hat K=\mathfrak G_T$  on $M$
generated by torus  $T$  and    the Weyl group   $\hat W$   of the maximal reductive subgroup ${\rm Aut}_r(M)$ of ${\rm Aut}(M)$ with respect to $T$.   Clearly,  \eqref{toric alpha} still holds for  $\alpha^{\mathfrak G_T}(M)$. Our result \eqref{alpha-Tn} shows that  $ \alpha^{T}(M)$ is also a limit of $\alpha^{T}_{m,k}(M,K_M^{-1})$ when $k\ge 2$.
\end{rem}

The following lemma gives a description for the set where $t(\cdot)$ attains its minimum $\alpha^{T}(M)$.

\begin{lem}\label{vert t-alpha}
Suppose that $p\in P$ satisfies $t(p)=\alpha^{T}(M)$ and $p\in\text{RelInt}(\mathcal F)$ is a point  in the relative interior of a face $\mathcal F$ of $P$.
Then, $$t(p')=\alpha^{T}(M),~\forall ~p'\in\mathcal F.$$
Consequently,
$$\{x\in  P|~t(x)=\alpha^{T}(M)\}=\cup_{i\in I}\mathcal F_i$$
for some faces $\mathcal F_i$ of $P$.
\end{lem}

\begin{proof}
Let $\hat {\mathcal F}$ be the affine span of $\mathcal F$ and  $q$ the intersection of the ray $\overrightarrow{pO}$ with $\partial P$. We consider the affine space  $\hat {\mathcal F}+q-p$.  For any $p'\in\mathcal F$, let $q'$ and  $q''$  be  the intersections  of the ray $\overrightarrow{p'O}$ with $\partial P$ and  $\hat {\mathcal F}+q-p$, respectively.  Then it is obvious that
\begin{eqnarray}\label{relation-p-q}
t(p)=\frac{|Oq|}{|pq|}=\frac{|Oq''|}{|p'q''|}.
\end{eqnarray}
We claim that
\begin{eqnarray}\label{hat F cap P}
\frac{t(p)}{t(p)-1}\mathcal F\subseteq\partial P.
\end{eqnarray}
Thus by (\ref{hat F cap P}), $q''=q'\in\partial P$ and the lemma is proved.

We prove (\ref{hat F cap P}) by a contradiction argument.  Since
$$\frac{t(p)}{t(p)-1}\mathcal F\subseteq(\hat {\mathcal F}+q-p),$$
 there will be two cases as following:

\begin{figure}[h]
  % Requires \usepackage{graphicx}
  \includegraphics[width=1.2in]{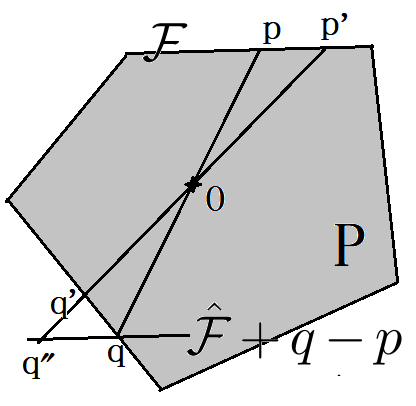}\\
  \caption{}%\label{}
\end{figure}

\emph{Case 1: $(\hat {\mathcal F}+q-p)\cap\text{Int}(P)\not=\emptyset$.} Since $p\in\text{RelInt}(\mathcal F)$, by the convexity of $P$, there exists a $p'\in\mathcal F$ near $p$ such that $q''\not\in  P$. Thus $q'$ lies between $O$ and $q''$ (see Figure 2). Consequently, by (\ref{relation-p-q}), we get
\begin{eqnarray}\label{contr}
t(p')=\frac{|Oq'|}{|p'q'|}<\frac{|Oq''|}{|p'q''|}=t(p)=\alpha^{T}(M),
\end{eqnarray}
which contradicts \eqref{toric alpha}. Thus this case is impossible.

\begin{figure}[h]
  % Requires \usepackage{graphicx}
  \includegraphics[width=1.2in]{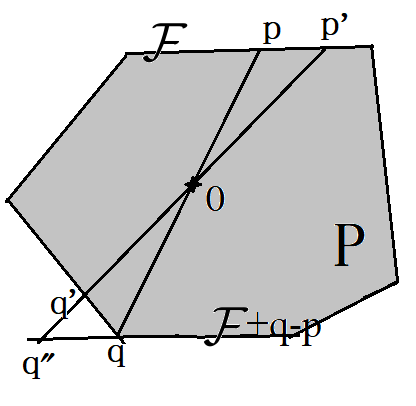}\\
  \caption{}%\label{}
\end{figure}

\emph{Case 2: $
(\hat {\mathcal F}+q-p)\cap  P\subseteq\partial P.$}  Suppose  that (\ref{hat F cap P}) is not true.  Then there are  $p'\in\mathcal F$  and $q''\in \overrightarrow{p'O}\cap (\hat {\mathcal F}+q-p)$  such that
$$q''=\frac{t(p)}{t(p)-1}p'\not\in {P}.$$
By the convexity of $P$,  we see that  $q'$ lies between $O$ and $q''$ (see Figure 3) and also \eqref{contr} holds. A contradiction to  \eqref{toric alpha}! Thus (\ref{hat F cap P}) holds and the lemma is proved.
\end{proof}

By using Proposition \ref{alpha m-k limit prop} and Lemma \ref{vert t-alpha}, we are able to prove Theorem \ref{weak conj thm}.

\begin{proof}[Proof of Theorem \ref{weak conj thm}]
Suppose that \eqref{number count} holds. Then   there is a  $T$-invariant subspace of dimension  $k$, $\Pi\in Gr_k^{T}(H^0(M,K^{-m_k}_M))$ such that
 $$I_\Pi=\{m_k\lambda_1,...,m_k\lambda_k\},$$
 where  $\{\lambda_1,...,\lambda_k\}\subseteq\left(\frac1{m_k}\mathfrak M\cap\mathcal F\right).$
 Thus by (\ref{alpha pi toric t}) and \eqref{toric alpha},  we have
  $$\alpha^\Pi_{m_k,k}=\alpha^{T}(M).$$
Note that for any $l\in\mathbb N_+$, the set  $I_{\Pi_l}=\{m_kl\lambda_1,...,m_kl\lambda_k\}$ determines a   $T$-invariant subspace of dimension  $k$,
$\Pi_l\in Gr_k^{T}(H^0(M,K^{-m_kl}_M)).$
Hence,  we also get
$$\alpha_{m_kl,k}^{\Pi_l}=\alpha^{T}(M).$$
Combining with \eqref{alpha-compare}, we obtain
\begin{eqnarray}
\alpha^{T}(M)\leq\alpha_{m_kl,k}^{T}(M,K_M^{-1})\leq\alpha_{m_kl,k}^{\Pi_l}=\alpha^{T}(M),\forall  ~l\in\mathbb N_+,
\end{eqnarray}
and \eqref{weak conj} is proved.

Conversely, we suppose  that \eqref{weak conj} holds. By \eqref{alpha toric t} there is some $\Pi\in Gr_k^{T}(H^0(M,K^{-m_k}_M))$ such that $\alpha^\Pi_{m_k,k}=\alpha^{T}(M)$. By \eqref{alpha pi toric t} and \eqref{toric alpha} we have
$$t|_{\hat I_\Pi}(x)\equiv \alpha^{T}(M).$$
Let $p_0\in\text{RelInt}(\hat  I_\Pi)$, then there is some face $\hat{\mathcal F}$ of $P$ such that $p_0\in\text{RelInt}(\hat{\mathcal F})$.
By Lemma \ref{vert t-alpha},
$$t|_{\hat {\mathcal F}}\equiv\alpha^T(M).$$
On the other hand, by convexity,
$$
\hat I_\Pi\subseteq\hat {\mathcal F}.
$$
This proves \eqref{number count}.
\end{proof}

\section{Examples}\label{sect exa}

In the section, we compute some examples and give precise value of  $\alpha^{K\times K}_{m,k}$-invariant    by using Theorem \ref{alpha m-1 K-K prop}  and Theorem \ref{weak conj thm}.   More precisely,  in Example \ref{ex2}, we will show that Conjecture \ref{Tian-conj}   holds  for any $k\ge 1$,   while Example \ref{ex} and \ref{ex-Pn} are
counter-examples of Conjecture \ref{Tian-conj} for $\alpha^T_{m,k}(M,K^{-1}_M)$-invariant on toric Fano manifolds when  $k\ge 2$. Example \ref{gl-exa} is a   Fano $GL_2(\mathbb C)$-compactification where  $\alpha^{K\times K}_{2l,1}(M,K^{-1}_M)$ and $\alpha^{K\times K}(M)$  are computed.

\begin{ex}\label{ex2}
Let $M$ be $\mathbb {CP}^1\times\mathbb{CP}^1$. It is a toric Fano manifold whose polytope $P$ is the convex hull of
$$P_1=(-1,1),P_2=(1,1),P_3=(1,-1),P_4=(-1,-1).$$
We see that $t(x)=\alpha^{T}(M)=\frac12$ for all $x\in\partial P$. It is easy to check \eqref{number count} holds for every $m\geq\left[\frac k2\right]+1$ in this case. Thus Conjecture \ref{Tian-conj} is confirmed for any $k\ge 1$.
\end{ex}

\begin{ex}\label{gl-exa}Let $M$ be  the toroidal Fano compactification of $GL_2(\mathbb C)$ constructed in \cite[Example 5.12]{Del3}, whose polytope $P_+$ is the convex hull of
$$P_1=\left(\frac12,\frac12\right),P_2=\left(2,-1\right),P_3=\left(2,-3\right),P_4=\left(-\frac12,-\frac12\right).$$
\begin{center}
\begin{tikzpicture}
\draw [dotted] (-1,-3) grid[xstep=1,ystep=1] (3,1);

\draw (0,0) node{$\bullet$};

\draw [semithick] (2,-1) -- (2,-3) -- (-1/2,-1/2) -- (1/2,1/2) -- (2,-1);
\draw (1.7,-2.2) node{$P_+$};

\draw [very thick, -latex] (0,0) -- (1,-1);

\draw [thick] (-0.5,-0.5) -- (1,-1) -- (2,-1.333);

\draw (1,-0.6) node{$2\rho$};
\draw (1,-1.2) node{$\overline O$};
\draw (0.7,0.6) node{$P_1$};
\draw (2.2,-1) node{$P_2$};
\draw (2.2,-3) node{$P_3$};
\draw (-0.7,-0.5) node{$P_4$};
\draw (2.2,-1.4) node{$Q_4$};
\draw (-0.2,0.2) node{$O$};
\end{tikzpicture}
\end{center}

By  \cite{Del3, LZZ}, $M$ admits a K\"ahler-Einstein metric. However, we will show that
$$\alpha^{K\times K}(M)=\frac25.$$
 which is strictly less than $\frac{4}{5}$,  the number in Tian's criterion for the existence of K\"ahler-Einstein metrics
on  4-dimensional Fano manifolds \cite{T87}.

The whole polytope $P$ is given by
$$\begin{aligned}P=\{(x,y)|~&l_1(x,y)=1+x+y\geq0,~l_2(x,y)=1-x-y\geq0,\\&l_3(x,y)=2-x\geq0,~l_4(x,y)=2-y\geq0\},\end{aligned}$$
and
$$P\cap{\mathfrak a_z^*}=\left\{(x,x)\left|-\frac12\leq x\leq\frac12\right.\right\}.$$
By  Corollary \ref{group alpha thm}, a direct computation  shows that
$$\alpha^{K\times K}(M)=\frac1{l_3(P_4)},$$
or equivalently,
$$\alpha^{K\times K}(M)=\frac{|\overline OQ_4|}{|Q_4P_4|},$$
where $\overline O=2\rho$.
Thus we may choose $v_{z_0}=P_4$ and $m_0=2$ in Theorem \ref{alpha m-1 K-K prop} and get
$$\alpha^{K\times K}_{2l,1}(M, K^{-1}_M)=\alpha^{K\times K}(M)=\frac25, ~\forall ~l\in N_{+}.$$
We remark that that $\alpha^{K\times K}(M)$ can also be computed by using \eqref{Del res}, which is proved in \cite{Del1}.

\end{ex}

\begin{ex}\label{ex}
Let $M$ be $\mathbb {CP}^2$ blown-up at $2$ points. It is a toric Fano manifold whose polytope is the convex hull of
$$P_1=(0,1),P_2=(1,0),P_3=(1,-1),P_4=(-1,-1),P_5=(-1,1).$$

\begin{center}
\begin{tikzpicture}
\draw [dotted] (-1,-1) grid[xstep=1,ystep=1] (1,1);

\draw (0,0) node{$\bullet$};

\draw[semithick]  (0,1) --(1,0)-- (1,-1)--(-1,-1)--(-1,1) -- (0,1);
\draw[thick]  (-1,-1) --(0.5,0.5);

\draw (0.2,1) node{$P_1$};
\draw (1.2,0) node{$P_2$};
\draw (1.2,-1) node{$P_3$};
\draw (-1.2,-1) node{$P_4$};
\draw (-1.2,1) node{$P_5$};
\draw (0.2,-0.2) node{$O$};
\draw (0.7,0.7) node{$Q_4$};
\end{tikzpicture}
\end{center}

We see that $\alpha^{T}(M)=t(P_4)=\frac{|OQ_4|}{|Q_4P_4|}=\frac13$. Moreover, $P_4$ is the only point which satisfies $t(\cdot)=\frac13$. By Theorem \ref{weak conj thm}, Conjecture \ref{Tian-conj} fails for every $k\in\mathbb N_{\geq2}$ in this case.
\end{ex}

\begin{ex}\label{ex-Pn}
Let $M$ be $\mathbb {CP}^n,n\geq2$. It is a toric Fano manifold whose polytope is the convex hull of
$$P_i=(-1,...,n,...,-1),i=1,...,n\text{ and } P_{n+1}=(-1,...,-1),$$
where $P_i$ has the $i$-th coordinate $n$ and others $-1$. We see that $\alpha^{T}(M)=t(P_\alpha)=\frac1{n+1}$ for all $\alpha=1,...,n+1$. Moreover, $P_\alpha$'s are the only points which satisfy $t(\cdot)=\frac1{n+1}$. By Theorem \ref{weak conj thm}, Conjecture \ref{Tian-conj} also fails for every $k\in\mathbb N_{\geq2}$ in this case.
\end{ex}

\section{Appendix: A direct proof of Lemma \ref{determinent}}

In this Appendix we  give a direct proof of (\ref{ricci-potential-1}) by an argument in  \cite[Section 5]{LTZ2}.   First we need the following lemma which  is essentially  a corollary of \cite[Theorem 9]{Timashev-Sbo}. \footnote{We would like to thank Professor D. A. Timash\"ev for telling us his paper \cite{Timashev-Sbo}.}

\begin{lem}\label{alpha-u-smoothness}
Suppose that $\alpha\in\Phi_+$ is a simple root and denote by $W_\alpha$ the Weyl wall
$$W_\alpha=\{y|~\alpha(y)=0\}.$$
Let $F$ be any facet of $P$, which is not orthogonal  to  $W_\alpha$ such that
\begin{align*}
F\cap W_\alpha\not=\emptyset,~F\cap \mathfrak a_+\not=\emptyset.
%F&\subseteq\{y|\alpha(y)\geq0\}
\end{align*}
Then the prime inner norm of $F$ satisfies
\begin{align}\label{inner-product-number}
\alpha(u)=-1.
\end{align}
\end{lem}

\begin{proof}
By assumption, there is a vertex $p_0$ of $F$ that lies in $W_\alpha$. Let $\mathfrak L$ be the Levi group of the parabolic subgroup $P(p_0)\subseteq G$ corresponding to $p_0$, and $\mathfrak Z$ be the closure of $\mathfrak L$ (called the transversal slice) defined in \cite[Section 8]{Timashev-Sbo}. Then there is a neighbourhood $\mathcal U$ of $p_0$, such that we can realize $\text{Span}_{\mathbb R_+}\{(\mathfrak a_+\cap\mathcal U)-p_0\}$ as the positive Weyl chamber $\mathfrak a_{\mathfrak L,+}$ of $\mathfrak L$, and $$\mathfrak C:=\text{Span}_{\mathbb R_+}\{P-p_0\}\cap\mathfrak M$$ as the cone generated by the weights of the $\mathfrak L$-linear representation $\mathbb V_0(-p_0)$ defining $\mathfrak Z$, respectively (cf. \cite[Proposition 6]{Timashev-Sbo}).

The smoothness criterion \cite[Theorem 9]{Timashev-Sbo} implies that $\mathfrak L$ is a product of general linear groups
$$\mathfrak L=\prod_{k=1}^{m_0}GL_{n_k}(\mathbb C),$$
and $\mathfrak Z$ is the product of the corresponding matrix algebras
$$\mathfrak Z=\prod_{k=1}^{m_0}\text{Mat}_{n_k\times n_k}(\mathbb C).$$
Thus $\mathfrak C$ is generated by all weights of $\cup_k\{\epsilon^k_1,...,\epsilon^k_{n_k}\}$, where $\{\epsilon^k_1,...,\epsilon^k_{n_k}\}$ are the weights of the natural $GL_{n_k}(\mathbb C)$-action on $\text{Mat}_{n_k\times n_k}(\mathbb C)$. Also the cone
$$\mathfrak C_+:=\mathfrak a_{\mathfrak L,+}\cap\mathfrak C$$ is generated by the weights (cf.  \cite[Section 11]{Timashev-Sbo})  $$\{\sum_{j=l}^l\epsilon^k_j|1\leq j\leq n_k\}.$$

On the other hand, by our assumption, we see that $$\mathfrak C_{F}:=\text{Span}_{\mathbb R_+}\{F-p_0\}$$ is a facet of $\mathfrak C$, which lies in the half-space $\{y|\alpha(y)\geq0\}$ and intersects $\mathfrak a_{\mathfrak L,+}$. The only possibility is that $\mathfrak C_{F}$ is generated by $$\cup_k\{\epsilon^k_1,...,\epsilon^k_{n_k}\}\setminus\{\epsilon^{k_0}_{n_{k_0}}\}$$
for some $k_0\in\{1,...,m_0\}$. In this way
$$\alpha =\epsilon^{k_0}_{n_{k_0}-1}-\epsilon^{k_0}_{n_{k_0}}
\text{ and }
u=(\epsilon^{k_0}_{n_{k_0}})^\vee.$$
Thus (\ref{inner-product-number}) is true and the lemma is proved.
\end{proof}

\begin{proof}[Proof of Lemma \ref{determinent}]
%We use an argument of \cite[Proposition 5.1]{LTZ2}.
Let $l_A(y)$ be the linear functions as in (\ref{linear-P}) and
$$\hat u_0(y)=\sum_Al_A\left(\frac12y\right)\log l_A\left(\frac12y\right),~y\in2P.$$
Let $\hat u$  be the Legendre function of $u$, defined by
$$\hat u(y(x))=\sum_i y_i x^i-u(x),$$
where $$y_i= \frac{\partial}{\partial x^i}u(x). $$  Then by \cite{Ab}, $\hat u$ can be regarded as a function on  $\overline{2P}$ and $(\hat u-\hat u_0)\in C^\infty(\overline{2P})$. Thus one can show that
\begin{align}\label{derivative-first}
\frac{\partial}{\partial y_i} \hat u&=\frac12\sum_{A=1}^{d_0}(u_A^i)(1+\log l_A)(y)+O(1)
%\frac{\partial^2}{\partial y_i\partial y_j}\hat u_{,ij}&=\frac12\sum_{A=1}^{d_0}\frac{u_{A}^i u_{A}^j}{l_A}(y)+O(1)
\end{align}
and
\begin{align}\label{derivative-second}\frac{\partial^2}{\partial y_i\partial y_j}\hat u_{,ij}=\frac12\sum_{A=1}^{d_0}\frac{u_{A}^i u_{A}^j}{l_A(y)}+O(1).
\end{align}

Define
$$h=2\Upsilon_{K^{-1}_M|_Z}(-x)+\log\mathbf J(x)-\log\det(\nabla^2u)-\log\prod_{\alpha\in\Phi_+}\langle\alpha,\nabla u\rangle^2(x).$$
It suffices to prove that $h$ is uniformly bounded on $\overline{\mathfrak a_+}$. Taking Legendre transformation, we have
\begin{align}\label{ricci-potential}
h=\log\det(\nabla^2\hat u)+2\Upsilon_{K^{-1}_M|_Z}(-\nabla\hat u(y))-\log\pi(y)+\log\mathbf J(\nabla\hat u(y)),\end{align}
where $\pi(y)=\prod_{\alpha\in\Phi_+}\alpha^2(y)$.
Since $\hat u$ is smooth in Int$(2P)$, $h$ is locally bounded in the interior of $2P_+$.  It  returns  to prove  that $h$ is bounded   near each $y_0\in\partial(2P_+)$. Suppose that there is a sequence such that
$$y_{(k)}\to y_0\in\partial(2P_+),~k\to+\infty.$$
By replacing $\{y_{(k)}\}_k$ by  a subsequence, we may assume that
$$x_{(k)}=\nabla\hat u(y_{(k)})\in(-\sigma_p),~k=1,2,...$$
for some  fixed vertex $p$  of $2P$ which lies in $\overline{\mathfrak a_+}$.

Recall the fact that on $(-\sigma_p)$,
\begin{eqnarray*}
\Upsilon |_{K^{-1}_M}(-x)=-\hat p(x),
\end{eqnarray*}
where $\hat p$ is determined by \eqref{upsilon-sigma-p}.
Then by (\ref{derivative-first}), we have
\begin{align}\label{upsi}
2\Upsilon_{K^{-1}_M|_Z}(-\nabla\hat u(y))=\sum_{A=1}^{d_0}(1-2\rho_A(u_A))\log l_A(y)+O(1).
\end{align}
Thus
plugging \eqref{upsi} into \eqref{ricci-potential} together with (\ref{derivative-first})-(\ref{derivative-second}), we get  (cf.  \cite[Section 5]{LTZ2}  for details), \footnote{ $u_A$ denotes the outer norm   in \cite{LTZ2} but here is the inner one}
\begin{align*}
h=&-2\sum_{A=1}^{d_0}\rho_A(u_A)\log l_A+2\sum_{\beta\in\Phi_+}\log\sinh(\frac12\sum_{A=1}^{d_0}\beta(u_A)\log l_A)\notag\\
&-2\sum_{\beta\in\Phi_+}\log\beta(y)+O(1)\notag\\
=&\sum_{\beta\in\Phi_+}\left[\sum_{A=1}^{d_0}|\beta(u_A)|\log l_A+2\log\sinh(\frac12\sum_{A=1}^{d_0}\beta(u_A)\log l_A)\right.\notag\\
&\left.-2\log\beta(y)\right]+O(1),~\text{as }y\to y_0.\notag
\end{align*}
%Here we a fact that
%$$\rho_A(u_A)=-\sum_{\beta\in\Phi_+}|\beta(u_A)|.$$

Denote
$$I_\beta(y)=\sum_{A=1}^{d_0}|\beta(u_A)|\log l_A+2\log\sinh(\frac12\sum_{A=1}^{d_0}\beta(u_A)\log l_A)-2\log\beta(y).$$
We need  to estimate each $I_\beta$.  As in \cite[Section 5]{LTZ2}, we divide  $y_0$  into  the following three cases:

\emph{Case-1.} $y_0\in\mathfrak a_+$ and is away from any Weyl walls;

\emph{Case-2.} $y_0\in\partial (2P_+)\setminus\partial(2P)$. That is $y_0$ lie on at least one Weyl wall but not on the boundary of $2P$;

\emph{Case-3.} $y_0\in\partial(2P)$ and lies on a Weyl wall $W_\alpha$.

\begin{figure}[h]
  \centering
  % Requires \usepackage{graphicx}
  \includegraphics[width=1.5in]{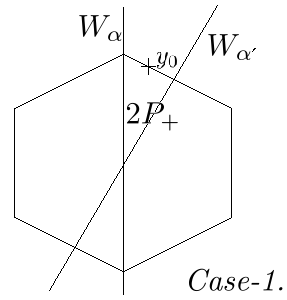}
  \includegraphics[width=1.5in]{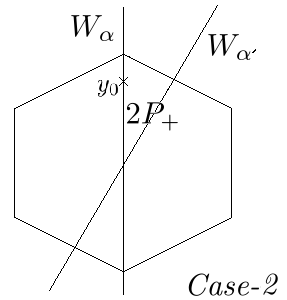}\\
  %\caption{}\label{}
\end{figure}

 In the first two cases, it is easy to show that
 \begin{align}\label{I-alpha}
I_\beta(y)=O(1),\text{ as }y\to y_0,~\forall~\beta\in\Phi_+.
\end{align}
In fact,  such two cases correspond  to  \emph{Case-1} and \emph{Case-2}   in \cite[Section 5]{LTZ2}, respectibvely.
Thus $h$ is bounded near $y_0$.

In  \emph{Case-3},   it  can be  direct to check that \eqref{I-alpha} holds for any $\alpha'\in\Phi_+$ such that  $y_0\not \in W_{\alpha'}$ as in \emph{Case-2}. Thus it suffices to estimate $I_\alpha$.  Suppose that there are $p$ facets $\{F_{A'}\}_{{A'}=1}^p$ of $2P$ passing through $y_0$. We have two subcases:

\emph{Case-3.1.} All of the $p$ facets are orthogonal to $W_\alpha$;

\emph{Case-3.2.} There is at least one pair of facets $\{F_1,F_2\}\subseteq\{F_{A'}\}_{{A'}=1}^p$ such that
$$F_2=s_\alpha(F_1),$$
which are not orthogonal to $W_\alpha$.

\begin{figure}[h]
  \centering
  % Requires \usepackage{graphicx}
  \includegraphics[width=1.5in]{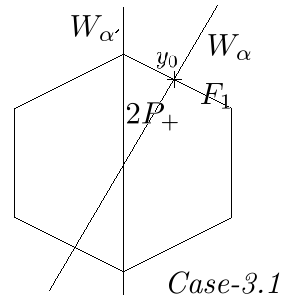}
  \includegraphics[width=1.5in]{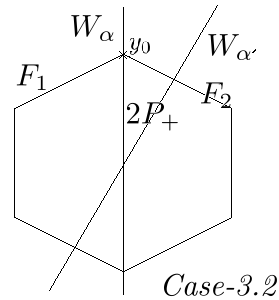}\\
  %\caption{}\label{}
\end{figure}

In \emph{Case-3.1}, \eqref{I-alpha} can be obtained as in  \emph{Case-3.1}  \cite[Section 5]{LTZ2}.

In \emph{Case-3.2}, we first show there is no more such pair of facets, which are not orthogonal to $W_\alpha$. In fact, if $\{F_{A'}\}_{{A'}=1}^{2p_1}\subseteq\{F_{A'}\}_{{A'}=1}^p$ are $2p_1$ facets passing through $y_0$, which are not orthogonal to $W_\alpha$ such that
$$F_{2k}=s_\alpha(F_{2k-1}),~k=1,...,p_1.$$
Then for the set of norms $\{u_{A'}\}$, we have
$$\text{rank}(u_1,...,u_p)=p-(p_1-1).$$
Thus $p_1=1$ by the Delzant condition.

Next we suppose that $F_2\cap\mathfrak a_+\not=\emptyset$  as in  \emph{Case-3.2} \cite[Section 5]{LTZ2}. Then  we see that
$$I_{\alpha}(y)=2(-\alpha(u_2)-1)\log l_2+O(1),~y\to y_0.$$
On the other hand, since $M$ is smooth, by Lemma \ref{alpha-u-smoothness}, we have
$$\alpha(u_2)=-1.$$
Thus,
$$I_\alpha(y)=O(1),\text{ as }y\to y_0.$$
Hence  \eqref{I-alpha} holds for any $\beta\in\Phi_+$.
 The  proof of lemma is complete.
\end{proof}

%%%%%%%%%%%%%%%%%%%%%%%%%%%%%%%%%%%%%%%%%%%%%%%%%%%%%%%%%%%%%%%%%%%%%%%%%%%%%%%%
%%%%%%%%%%%%%%%%%%%%%%%%%%%%%%%%%%%%%%%%%%%%%%%%%%%%%%%%%%%%%%%%%%%%%%%%%%%%%%%%
%%%%%%%%%%%%%%%%%%%%%%%%%%%%%%%%%%%%%%%%%%%%%%%%%%%%%%%%%%%%%%%%%%%%%%%%%%%%%%%%

\end{document}